\documentclass[12pt]{amsart}
\usepackage{color,graphicx,enumerate,wrapfig,amssymb,mathtools}
\usepackage{pxfonts}
\usepackage{subfigure}
\usepackage{appendix}
\mathtoolsset{showonlyrefs}
\usepackage{eucal}
\usepackage{esvect}

\newcommand{\ddiv}{{\mbox{div}}}

\newcommand{\R}{{\mathbb R}}
\newcommand{\sS}{{\mathbb S}}
\newcommand{\Z}{{\mathbb Z}}

\newcommand{\cM}{{\mathcal M}}
\newcommand{\cL}{{\mathcal L}}
\newcommand{\cS}{{\mathcal S}}

\newcommand{\cH}{{\mathcal H}}
\newcommand{\cN}{{\mathcal N}}

\newcommand{\e}{\varepsilon}

\newcommand{\dist}{\operatorname{dist}}

\newcommand{\p}{\partial}

\newcommand{\ra}{\rightarrow}

\newcommand{\norm}[1]{\left\Arrowvert {#1}\right\Arrowvert}

\theoremstyle{plain}
\newtheorem{theorem}{Theorem}[section]

\newtheorem{corollary}[theorem]{Corollary}

\newtheorem{lemma}[theorem]{Lemma}

\theoremstyle{definition}
\newtheorem{definition}{Definition}

\theoremstyle{remark}

\newtheorem*{case*}{Case}

\newtheorem*{claim*}{Claim}

\newtheorem{notation}{Notation}
\newtheorem{remark}{Remark}

\numberwithin{equation}{section}

\title [Nodal Sets for ``Broke'' PDEs]{Nodal Sets for ``Broken'' Quasilinear PDEs}
\author{Sunghan Kim}
\address{Department of Mathematical Sciences, Seoul National University, Seoul 08826, Korea}
\email{sunghan290@snu.ac.kr}

\author{Ki-Ahm Lee}
\address{Department of Mathematical Sciences, Seoul National University, Seoul 08826, Korea
\& Korea Institute for Advanced Study, Seoul 02455, Korea}
\email{kiahm@snu.ac.kr}

\author{Henrik Shahgholian}
\address{Department of Mathematics, Royal Institute of Technology,
  100~44  Stockholm, Sweden}
\email{henriksh@kth.se}
\thanks{S. Kim has been supported by National Research Foundation of Korea (NRF) grant funded by the Korean government (NRF-2014-Fostering Core Leaders of the Future Basic Science Program). K.-A. Lee has been supported by NRF grant funded by the Korean government (MSIP) (No. NRF-2015R1A4A1041675). K.-A. Lee also holds a joint appointment with the Research Institute of Mathematics of Seoul National University. H. Shahgholian has been supported in part by Swedish Research Council.
}

\begin{document}

\begin{abstract}
We study the local  behavior of the nodal sets of the solutions to elliptic quasilinear equations with nonlinear conductivity part,
\begin{equation*}
\ddiv(A_s(x,u)\nabla u)=\ddiv\vv{f}(x),
\end{equation*} 
where $A_s(x,u)$ has ``broken'' derivatives of order $s\geq 0$, such as 
\begin{equation*}
A_s(x,u) = a(x) + b(x)(u^+)^s,
\end{equation*}
with $(u^+)^0$ being understood as the characteristic function on $\{u>0\}$. The vector $\vv{f}(x)$ is assumed to be $C^\alpha$ in case $s=0$, and $C^{1,\alpha}$ (or higher) in  case $s>0$.

Using geometric methods, we prove almost complete  results (in analogy with  standard PDEs) concerning the behavior of the nodal sets. More exactly, we show that the nodal sets, where solutions have (linear) nondegeneracy, are locally smooth graphs.  Degenerate points are shown to have  structures that follow the lines of arguments as that of  the nodal sets for harmonic functions, and general PDEs.
\end{abstract}

\maketitle

\tableofcontents

%%%%%%%%%%%%%%%%%%%%%%%
%
% Section: Introduction
%
%%%%%%%%%%%%%%%%%%%%%%%

\section{Introduction}

The  regularity issue of nodal sets of solutions to ``broken'' elliptic quasilinear equations is the subject matter of this paper. The model equation of concern has the form
\begin{equation*}
\ddiv(A_s(x,u)\nabla u)=\ddiv\vv{f}(x),
\end{equation*} 
where $A_s(x,u)$ will  have a structure of type 
\begin{equation*}
A_s(x,u) =a (x) +b(x) (u^+)^s,
\end{equation*}
with $s\geq 0$.  The particular case, $s=0$, is understood as $(u^+)^0 = H(u)$ with $H$ being the Heaviside function. Since our interest is in the local behavior of solutions, we shall in general not specify any domains or boundary values, unless it becomes necessary. Nevertheless, the readers may consider this equation in the unit ball  (or their favorite domain) with reasonable boundary values. 

%%%%%%%%%%%%%%%%%%%%%%%
% Subsection: Background
%%%%%%%%%%%%%%%%%%%%%%%

\subsection{Background}

The study of nodal sets for elliptic equations has been of great interest to the fields of geometry and analysis. The study mainly involves local finiteness of nodal set in $(n-1)$-dimensional Hausdorff measure, the structure theorem of singular set $\{u=|\nabla u|=0\}$ and its local finiteness in $(n-2)$-dimensional Hausdorff measure. Recent development in this framework can be found in \cite{CF}, \cite{L}, \cite{H}, \cite{HS}, \cite{HHN}, \cite{HHHN}, \cite{B}, \cite{A}, \cite{CNV} and the references therein.

Let us remark that Lipschitz or higher regularity of the coefficients was essential in the existing literature, which is largely because of the necessity of a rigid control over vanishing order. Indeed, it is shown in \cite{GL} that vanishing order is bounded locally uniformly in general when coefficients are Lipschitz; also see \cite{M} and \cite{P} for counterexamples in this regard when coefficients are only H\"{o}lder continuous.

This paper, however, provides a structure theorem for the nodal sets of the solutions to ``broken'' partial differential equations (PDEs in the sequel), and analyze their lower dimensional Hausdorff measure, where the coefficient $A_s(x,u)$ experience discontinuous ($s=0$) or H\"{o}lder continuous ($0<s<1$) conductivity jump; see Theorem \ref{theorem:main-deg} and Theorem \ref{theorem:hauss-h}. In any case, the associated regularity is strictly weaker than Lipschitz continuity. To the best of the authors' knowledge, this is the first work concerning non-Lipschitz coefficients. 

It is worthwhile to mention that the ``break'' in conductivity appears in several applications, related to composite materials.  The peculiarity of our model is that it deals with materials that undergo phase transitions, due to certain thresholds of physical/chemical quantity, that in turn implies jump in conductivity. We refer to our previous work \cite{KLS} in this regard. This contrasts to the artificiality of the counterexamples constructed in \cite{M} and \cite{P} for non-Lipschitz coefficients, asking for a more sophisticated analysis on nodal sets arising from discontinuous (or less regular) phenomena in the real world. 

%%%%%%%%%%%%%%%%%%%%%%%
% Subsection: Heuristic Arguments
%%%%%%%%%%%%%%%%%%%%%%%

\subsection{Heuristic Arguments}

In this paper, we prove the optimal regularity of solutions and the (higher) regularity of the nodal set around nondegenerate points. More exactly, we obtain the optimal $C^{0,1}$ regularity when $x\mapsto A_s(x,u(x))$ is only bounded (the case $s=0$), and the optimal $C^{m,1}$ regularity when $x\mapsto A_s(x,u(x))$ is $C^{m-1,1}$ (the case $s=m$ for some integer $m\geq 1$); see Theorem \ref{theorem:main} (i) and respectively Theorem \ref{theorem:reg-h} (ii). To the best of our knowledge, such a regularity theory (for solutions) is considered only in our previous work \cite{KLS}, where the authors treated the same problem with zero right hand side. Though the techniques in the authors' previous paper can be generalized to the case of bounded right hand side, they fail to hold in the framwork of the current paper. 
 
At technical level, our approach is based on the idea of the Schauder theory, that is, the approximation of solutions with polynomials. Indeed, by simple transformation of the equation, at a given point $z \in \partial \{u>0\}$, through  $v_z(x)=\phi_z (u (x) ) $ for some $\phi_z$\footnote{The exact form of this function is $\phi_z(u)(x)=a_+(z)u^+(x)-a_-(z)u^-(x)$. See Definition \ref{definition:phi} for more details.} that is naturally given by the problem we arrive at $\Delta v_z(x) = g_z(x)$, where $g_z(x) $ now is smooth enough at the point $z$, but not (necessarily) at other free boundary points.  This in turn gives us that $v_z(x)$ is  $C^{1,\beta}$ for $\beta < \alpha$ at the point $z$, with uniform norms. Since the  norms are uniform with respect to any   free boundary point, we can conclude the Lipschitz regularity of the solution $u$ to our problem. A further, and straightforward conclusion is that the nondegenerate points of the  free boundary is squeezed between $C^{1,\beta}$ graphs. This implies the regularity of the free boundary at nondegenerate points (see Theorem \ref{theorem:main}).  

When a free boundary point is degenerate, that is, when the solution vanishes with order greater than $1$, we use a bootstrap argument to prove that the vanishing order is indeed an integer (Lemma \ref{lemma:u-d}). This result extends the integral vanishing order studied in \cite{B} and \cite{H} into the regime of broken PDEs. Now that we have a control on the vanishing order of the set of degenerate points, we follow the idea of \cite{H} and \cite{HS} and derive the structure theorem and the Hausdorff dimension of the singular set.

It should be stressed that Almgren's frequency formula does not hold for the broken PDEs under our consideration. Roughly speaking, this is because the governing PDE cannot be derived by the variational approach. Such a drawback prevents us from achieving the local finiteness of the $(n-1)$-dimensional Hausdorff measure of the entire nodal set, $\{u=0\}$, unless the coefficients $a_+$ and $a_-$ are $C^{1,1}$. In the case that $a_+$ and $a_-$ are $C^{1,1}$, however, we observe that the solution $u$ can be transformed to a function that solves an elliptic equation with bounded lower order terms with the nodal set being the same as that of $u$. Hence, the classical theory can be applied to the transformed one, and the desired result (Theorem \ref{theorem:main-deg} (ii)) follows.

Let us address some interesting features on the case $s<0$, which is not considered here. Even the existence of solutions seems to be a difficult problem, when the coefficients $a$ and $b$ are not constants. To the best of the authors' knowledge, this 
problem has not been studied in the existing literature. We invite the readers to have a look at  Remark \ref{remark:s<0} for further discussions. 

%Moreover, when $s\leq -1$, we think the corresponding \bblue{PDE may only have positive solutions. } \mpar{Do we think so? Not sure}Furthermore, the solutions may blow up as they approach their nodal sets (from the positive side), whence we obtain a set of singularities. 
%However, these assertions need to be rigorously justified, and we ask the reader to study this problem if interested.

%%%%%%%%%%%%%%%%%%%%%%%
% Subsection: Outline
%%%%%%%%%%%%%%%%%%%%%%%

\subsection{Outline}

The paper is organized as follows. In Section \ref{section:zero}, we study the broken PDEs when the order of the ``break'' is $0$; that is, when the operator $A_s(x,u)$ is given with $s=0$. The main results concerning these equations are listed in Subsection \ref{subsection:elliptic}. Subsection \ref{subsection:optimal reg} and Subsection \ref{subsection:C1a-fb} are devoted to studying the nondegenerate part of the nodal set and proving Theorem \ref{theorem:main}. Subsection \ref{subsection:deg} and Subsection \ref{subsection:hausdorff} concern the structure of the singular part of the nodal set, and we prove Theorem \ref{theorem:main-deg}. 

In Section \ref{section:higher}, we study the broken PDEs when the order of the ``break'' is positive. Subsection \ref{subsection:reg-h} is devoted to studying the optimal regularity of solutions and proving Theorem \ref{theorem:reg-h}. In Subsection \ref{subsection:hausdorff-h} we are concerned with the local finiteness of the $(n-1)$-dimensional Hausdorff measure of the nodal set and prove Theorem \ref{theorem:hauss-h}.

In Appendix, we state and prove some variants of the Schauder theory and 
%\bblue{ Almgren's frequency formula that are used in this paper.}
Almgren's frequency formula, which is used for the exceptional case $a_+,a_-\in C^{1,1}$.
 We contain this part for the reader's convenience.

%%%%%%%%%%%%%%%%%%%%%%%
%
% Section: Broken PDEs with Break of Order $s=0$
%
%%%%%%%%%%%%%%%%%%%%%%%

\section{Broken PDEs with Order $s=0$}\label{section:zero}

%%%%%%%%%%%%%%%%%%%%%%%
% Subsection: Problem Setting and Main Results
%%%%%%%%%%%%%%%%%%%%%%%

\subsection{Problem Setting and Main Results}\label{subsection:elliptic}

Let $u$ be a weak solution of 
\begin{equation}\tag{$L$}\label{eq:main}
\ddiv(A(x,u)\nabla u)=\ddiv\vv{f}(x)\quad\text{in }B_1,
\end{equation}
where $\vv{f}$ is a vector-valued function and 
\begin{equation}\label{eq:simple}
A(x,u)=a_+(x)H(u)+a_-(x)(1-H(u)),
\end{equation}
with $H$ being the Heaviside function and $a_+$ and $a_-$ being real-valued functions satisfying 
\begin{equation}\label{eq:ellip}
\lambda\leq a_\pm(x)\leq \frac{1}{\lambda}\quad\text{and}\quad \left|\vv{f}(x)\right|\leq\frac{1}{\lambda}\quad (x\in \bar{B}_1),
\end{equation}
with a constant $0<\lambda<1$ and 
\begin{equation}\label{eq:cont}
\max\left\{|a_\pm(x)-a_\pm(y)|,\left|\vv{f}(x)-\vv{f}(y)\right|\right\}\leq \omega_0|x-y|^\alpha\quad (x,y\in \bar{B}_1),
\end{equation}
for some $\omega_0>0$ and $0<\alpha<1$.

Let us first introduce some notations in regards to vanishing order of solutions at free boundary points.
\begin{definition}\label{definition:fb} Let $v$ be a continuous function. 

\begin{enumerate}[(i)]
\item By $\Gamma(v)$ we denote $\{x\in B_1:v(x)=0\}$. By $\Gamma_*(v)$ denote the subset of $\Gamma(v)$ consisting of vanishing points with finite orders, i.e.,
\begin{equation*}
\Gamma_*(v)=\left\{z\in\Gamma(v):\limsup_{x\ra z}\frac{|v(x)|}{|x-z|^d}=\infty\text{ for some $d>0$}\right\}.
\end{equation*}
\item $\cN(v)$ is defined by the class of vanishing points with order 1, namely, 
\begin{equation*}
\cN(v)=\left\{z\in\Gamma(v):\limsup_{x\ra z}\frac{|v(x)|}{|x-z|}>0\right\}.
\end{equation*}
The points in $\cN(v)$ will be called {\it nondegenerate} points (of $v$). In particular, whenever $z\in\cN(v)$, we will call the associated positive number, $\limsup_{x\ra z}|x-z|^{-1}|v(x)|$, the nondegeneracy constant of $v$ at $z$.

\item By $\cS(v)$ we denote $\Gamma_*(v)\setminus\cN(v)$, which contains all vanishing points with finite orders strictly larger than 1. Elements in $\cS(v)$ will be called  (polynomially) degenerate points. 
\end{enumerate}
\end{definition}

It is not hard to check that problem \eqref{eq:main} admits a weak solution, given a reasonable boundary data. For a detailed proof, see \cite{KLS}. 

\begin{definition}\label{definition:phi} Given $z\in\Gamma(u)$, define a function $\phi_z(u)$ by
\begin{equation*}
\phi_z(u)(x)=a_+(z)u^+(x)-a_-(z)u^-(x).
\end{equation*}
We will simply denote this function by $\phi(u)$ when $z=0$. 
\end{definition}

\begin{remark}\label{remark:phi}
Observe that $\phi_z(u)$ is a weak solution to 
\begin{equation}\label{eq:phi}
\Delta \phi_z(u)=\ddiv(\vv{f}(x)+\sigma_z(x)\nabla u)\quad\text{in }B_1,
\end{equation}
where 
\begin{equation}\label{eq:sigma}
\sigma_z(x)=(a_+(z)-a_+(x))H(u)+(a_-(z)-a_-(x))(1-H(u)).
\end{equation}
It is important to see that 
\begin{equation}\label{eq:sigma-Ca}
|\sigma_z(x)|\leq\omega_0|x-z|^\alpha \quad (x\in \bar{B}_1).
\end{equation} 
\end{remark}

The expert readers may already have noticed that the definition of $\phi_z(u)$, suggests the approach of  freezing method as in classical PDEs, and pointwise  Schauder theory.
%Let us state our main results. 
Hence $\phi_z(u)$ plays a central role in proving our main theorems, together with the pointwise regularity for elliptic PDEs presented in Appendix \ref{appendix:ptwise}. 

The first result of our paper is the Lipschitz regularity of solutions and the $C^{1,\alpha}$ regularity of free boundaries near nondegenerate points. 

\begin{theorem}\label{theorem:main} Let $u$ be a nontrivial weak solution to \eqref{eq:main} with $|u|\leq 1$ in $B_1$. Then the following hold.
\begin{enumerate}[(i)]
\item $u\in C^{0,1}(B_{1/2})$ with uniform norm
\begin{equation*}
\norm{\nabla u}_{L^\infty(B_{1/2})}\leq C,
\end{equation*}
where $C$ is a positive constant depending only on $n$, $\lambda$, $\omega_0$ and $\alpha$. 
\item $\cN(u)$ is locally a $C^{1,\alpha}$ graph in the sense that for any $z\in\cN(u)$, there exist $r_0>0$ and $C_0>0$, which depend only on $n$, $\lambda$, $\omega_0$, $\alpha$ and the nondegeneracy constant of $u$ at $z$ such that
\begin{equation*}
|\nu(x)-\nu(y)|\leq C_0|x-y|^\alpha\quad (x,y\in B_{r_0}(z)\cap\Gamma(u)),
\end{equation*}
where $\nu$ is the unit normal of $\Gamma(u)$ pointing inward $\{u>0\}$.
\end{enumerate}
\end{theorem}

The second result concerns a structure theorem for degenerate free boundary points. Here we denote by $\dim_\cH(S)$ the Hausdorff dimension of a set $S$, and by $\cH^r(S)$ the $r$-dimensional Hausdorff measure of $S$.

\begin{theorem}\label{theorem:main-deg} Suppose that $u$ is a nontrivial weak solution to \eqref{eq:main} in $B_1$, when $\vv{f}\equiv 0$.  
\begin{enumerate}[(i)]
\item Suppose that the vanishing order of $u$ is finite in $B_1$. Then
\begin{equation*}
\cS(u)=\displaystyle\bigcup_{j=0}^{n-2}\cM^j(u),
\end{equation*} 
where $\cM^j(u)$ is contained in a countable union of $j$-dimensional $C^1$ manifolds, for each $j=0,1,\cdots,n-2$. In particular, 
\begin{equation*}
\dim_\cH (\cS(u)) \leq n-2.
\end{equation*}
\item Suppose that $a_+,a_-\in C^{1,1}(\bar{B}_1)$. Then there is a positive constant $C$ depending only on $n$, $\lambda$ and $\norm{a_\pm}_{C^{1,1}(B_1)}$ such that 
\begin{equation*}
\cH^{n-1}(\Gamma(u)\cap B_{1/2})\leq CN,
\end{equation*}
where \footnote{Due to the maximum principle, a nontrivial solution $u$ to \eqref{eq:main} cannot vanish entirely on $\p B_1$, which assures that $N$ is always finite under the circumstances of Theorem \ref{theorem:main-deg}.} 
\begin{equation*}
N = \frac{\int_{B_1}|\nabla u|^2}{\int_{\p B_1}u^2}.
\end{equation*}
\end{enumerate}
\end{theorem}

%%%%%%%%%%%%%%%%%%%%%%%
% Subsection: Optimal Regularity of Solutions
%%%%%%%%%%%%%%%%%%%%%%%

\subsection{Optimal Regularity of Solutions}\label{subsection:optimal reg}

We begin with interior H\"{o}lder regularity of solutions to \eqref{eq:main} across their free boundaries. One may find the same proof in \cite{KLS}, though we provide it for the sake of completeness.

\begin{lemma}\label{lemma:Ca} For each $0<\beta<1$, there exist $C_\beta,r_\beta>0$, depending only on $n$, $\lambda$, $\omega_0$, $\alpha$ and $\beta$, such that
\begin{equation}\label{eq:Ca}
|u(x)|\leq C_\beta|x-z|^\beta,
\end{equation} 
for any $x\in B_{r_\beta}(z)$ and any $z\in\Gamma(u)\cap B_{1/2}$.
\end{lemma}

\begin{proof} Suppose towards a contradiction that there are a sequence $\{u_j\}_{j=1}^\infty$ of solutions to \eqref{eq:main} with $|u_j|\leq 1$ in $B_1$, a sequence $\{z^j\}_{j=1}^\infty\subset\Gamma(u)\cap B_{1/2}$ and a sequence $\{r_j\}_{j=1}^\infty$ of radii which decrease to zero as $j\ra\infty$ such that $\sup_{B_{r_j}(z^j)}|u_j|=j$ and $\sup_{B_r(z^j)}|u_j|\leq jr^\beta$ for any $r_j\leq r\leq 1$. For each $j=1,2,\cdots$, define a function $\tilde{u}_j$ on $B_{1/r_j}$ by $\tilde{u}_j(x) =\frac{u_j(r_jx+z^j)}{jr^\beta}$. Then 
\begin{equation}\label{eq:false-assump}
\sup_{B_1}|\tilde{u}_j|=1\quad\text{and}\quad \sup_{B_R}|\tilde{u}_j|\leq R^\beta\quad\text{for any }1\leq R\leq \frac{1}{r_j}.
\end{equation}

As $\{z^j\}_{j=1}^\infty$ being a bounded sequence, there exists a limit point $z^0\in\bar{B}_{1/2}$ to which $\{z^j\}_{j=1}^\infty$ converges possibly along a subsequence; we are going to denote this subsequence simply by $\{z^j\}_{j=1}^\infty$ for notational convenience. 

Observe that $\{\tilde{u}_j\}_{j=1}^\infty$ is locally uniformly bounded in $\R^n$, and that each $\tilde{u}_j$ is a weak solution to 
\begin{equation}\label{eq:tuj}
\ddiv(A_j(x,\tilde{u}_j)\nabla \tilde{u}_j)=\ddiv(\vv{f_j}(x))\quad\text{in }B_{1/{2r_j}},
\end{equation}
where $A_j(x,s)=A(r_jx+z^j,s)$ and $\vv{f_j}(x)=\vv{f}(r_jx+z^j)$; notice that $A_j$ and $f_j$ satisfy \eqref{eq:ellip} with constant $\lambda$ uniformly for $j=1,2,\cdots$. Hence, the standard interior H\"{o}lder regularity theory and the interior energy estimate apply to $\{\tilde{u}_j\}_{j=1}^\infty$, resulting a function $\tilde{u}_0\in C_{loc}^\gamma(\R^n)\cap W_{loc}^{1,2}(\R^n)$, with $\gamma$ being universal, such that $\tilde{u}_j\ra \tilde{u}_0$ uniformly in $C_{loc}^{\gamma'}(\R^n)$ for any $\gamma'<\gamma$ as well as that $\nabla\tilde{u}_j\ra \nabla\tilde{u}_0$ weakly in $L^2(\R^n)$. In addition, it is not hard to see that $A_j(x,\tilde{u}_j)\nabla\tilde{u}_j\ra A_0(\tilde{u}_0)\nabla\tilde{u}_0$ weakly in $L^2(\R^n)$ as well, where $A_0(\tilde{u}_0)=a_+(z^0)H(\tilde{u}_j)+a_-(z^0)(1-H(\tilde{u}_0))$. Moreover, the continuity of $\vv{f}$ implies that $\vv{f_j}\ra \vv{f}(0)$ locally uniformly in $\R^n$ and $\ddiv\vv{f_j}\ra 0$ weakly star in $L_{loc}^\infty(\R^n)$. Passing to the limit in \eqref{eq:false-assump} and \eqref{eq:tuj}, we obtain 
\begin{equation*}
\sup_{B_1}|\tilde{u}_0|=1\quad\text{and}\quad \sup_{B_R}|\tilde{u}_0|\leq R^\beta\quad\text{for any }R\geq 1,
\end{equation*}
as well as that $\tilde{u}_0$ solves the following PDE, in the sense of distribution, 
\begin{equation*}
\ddiv(A_0(\tilde{u}_0)\nabla\tilde{u}_0)=0\quad\text{in }\R^n.
\end{equation*}

Define a function $w_0$ on $\R^n$ by $w_0(x)= a_+(z^0)\tilde{u}_0^+(x)-a_-(z^0)\tilde{u}_0^-(x)$. Then $w_0\in C_{loc}^\gamma(\R^n)\cap W_{loc}^{1,2}(\R^n)$, and due to the convergence results observed right above, $w_0$ becomes a harmonic function in the entire space, satisfying $\sup_{B_1}|w_0|\geq\lambda$ and $|w_0(x)|\leq \lambda^{-1}|x|^\beta$ for $|x|\geq 1$. According to the Liouville theorem, such a harmonic function must be constant. However, since $w_0(0)=0$ as $\tilde{u}_0(0)=0$, $w_0$ must vanish in the whole space, a contradicting the fact that  $\sup_{B_1}|\tilde{u}_0|=1$.
\end{proof}

A straightforward  consequences  to the preceding lemma  is
 a rough estimate of the growth rate of the gradient near a free boundary point,  which in turn implies 
the  $L^2$-BMO type estimate for the gradient.
 The proofs are standard but included for the sake of completeness.

\begin{corollary}\label{corollary:more-W12} For each $0<\delta<\alpha$, with $\alpha$ chosen from \eqref{eq:cont}, there exist $C_\delta, \rho_\delta>0$ depending only on $n$, $\lambda$, $\omega_0$, $\alpha$ and $\delta$, such that 
\begin{equation}\label{eq:control-grad u}
r^\delta\left(\fint_{B_r(z)}|\nabla u|^2\right)^{1/2}\leq C_{\delta},
\end{equation}
for any $z\in\Gamma(u)\cap B_1$ and any $0<r\leq \rho_\delta$.,
\end{corollary}

\begin{proof} Let $z\in\Gamma(u)\cap B_{1/2}$ and choose $\rho_\delta$ by $r_\beta$ in Lemma \ref{lemma:Ca} with $\beta=1-\delta$. Set $u_{z,r}$ as a function on $B_1$ defined by $u_{z,r}(x)=\frac{u(rx+z)}{r^{1-\delta}}$. Applying \eqref{eq:Ca} with $\beta=1-\delta$, we obtain $|u_{z,r}|\leq C_\delta$ in $B_{3/4}$, for some positive constant $C_\delta$, depending only on $n$, $\lambda$, $\omega_0$, $\alpha$ and $\delta$. As $u_{z,r}$ satisfies a PDE whose coefficient also satisfies the ellipticity condition in \eqref{eq:ellip}, we deduce from the local energy estimate that 
\begin{equation*}
r^{2\delta}\fint_{B_r(z)}|\nabla u|^2=\fint_{B_{1/2}}|\nabla u_{z,r}|^2\leq C\fint_{B_{3/4}}u_{z,r}^2\leq CC_\delta^2
\end{equation*}
where $C>0$ depends only on $n$ and $\lambda$.
\end{proof}

Next we prove $L^2$-BMO type estimates for the gradient of $u$ at its free boundary points. 

\begin{notation}\label{notation:avg} In what follows, we shall use $(v)_{z,r}$ to denote $(v)_{z,r} = \fint_{B_r(z)}v$, namely the average of $v$ over $B_r(z)$, where $z\in\R^n$, $r>0$ and $v\in L^1(B_r(z))$. 
\end{notation}

\begin{lemma}\label{lemma:L2-BMO} There exist $C_0,r_0>0$, depending only on $n$, $\lambda$, $\omega_0$ and $\alpha$, such that 
\begin{equation}\label{eq:L2-BMO}
\fint_{B_r(z)}|\nabla u-(\nabla u)_{z,r}|^2\leq C_0,
\end{equation}
for any $0<r\leq r_0$ and any $z\in\Gamma(u)\cap B_{1/2}$. 
\end{lemma}

\begin{proof} Fix $z\in\Gamma(u)\cap B_{1/2}$. Since $u=\frac{1}{a_+(z)}(\phi_z(u))^+-\frac{1}{a_-(z)}(\phi_z(u))^-$, it is enough to prove \eqref{eq:L2-BMO} with $\phi_z(u)$ instead of $u$. 

Choose $\delta=\alpha/2$ in Corollary \ref{corollary:more-W12}. Then for any $0<r<r_1$,
\begin{equation*}
\fint_{B_r(z)}|\sigma_z\nabla u|^2 \leq \omega_0r^{2\alpha}\fint_{B_r(z)}|\nabla u|^2\leq C_1r^\alpha,
\end{equation*}
where $\sigma_z$ is the function defined by \eqref{eq:sigma} and $r_1$ and $C_1$ are positive constants depending only on $n$, $\lambda$, $\omega_0$ and $\alpha$. Now we choose $0<r_0\leq r_1$ such that $C_1r_0^\alpha\leq \eta^2$, where $\eta$ is the constant chosen from Lemma \ref{lemma:bmo}. Since $\phi_z(u)$ solves \eqref{eq:phi}, the rest of the proof is a direct consequence of Lemma \ref{lemma:bmo} with $\vv{f}$ replaced by $\vv{f}+\sigma_z\nabla u$ and a canonical scaling argument. 
\end{proof}

By means of the John-Nirenberg inequality, we may now improve Corollary \ref{corollary:more-W12} for $L^p$ norms with $1\leq p<\infty$.

\begin{lemma}\label{lemma:more-W1p} Let $1\leq p<\infty$. Then there exists $C_p>0$, depending only on $n$, $\lambda$, $\omega_0$, $\alpha$ and $p$ such that $u\in W^{1,p}(B_{r_0}(z))$ for any $z\in\Gamma(u)\cap B_{1/2}$ and 
\begin{equation}\label{eq:W1p}
\norm{u}_{W^{1,p}(B_{r_0})}\leq C_p,
\end{equation}
where $r_0$ is the constant chosen from Lemma \ref{lemma:L2-BMO}. Moreover, for each $0<\delta<\alpha$, there exist $C_\delta, r_\delta>0$, depending only on $n$, $\lambda$, $\omega_0$, $\alpha$ and $\delta$, such that 
\begin{equation}\label{eq:control-grad u-W1p}
r^\delta\left(\fint_{B_r(z)}|\nabla u|^p\right)^{1/p}\leq C_\delta,
\end{equation}
for any $z\in\Gamma(u)\cap B_1$, any $0<r\leq r_\delta$.
\end{lemma}

\begin{proof} We shall prove the second part of this lemma, since the first part uses a similar argument. Let $\delta\in(0,\alpha)$ be given. Notice that \eqref{eq:control-grad u} and the H\"{o}lder inequality imply
\begin{equation*}
|(\nabla u)_{z,r}|\leq \left(\fint_{B_r(z)}|\nabla u|^2\right)^{\frac{1}{2}}\leq C_\delta r^{-\delta},
\end{equation*}
for any $0<r<\rho_\delta$, where $C_\delta$ and $\rho_\delta$ are chosen as in \eqref{eq:control-grad u}. On the other hand, the John-Nirenberg inequality together with \eqref{eq:L2-BMO} yields that 
\begin{equation*}
\norm{\nabla u - (\nabla u)_{z,r}}_{L^p(B_r(z))}\leq C_0|B_r|^{\frac{1}{p}},
\end{equation*}
whenever $0<r\leq r_0$ with $C_0$ and $r_0$ chosen from Lemma \ref{lemma:L2-BMO}. Putting these two inequalities together, we derive for any $0<r<\min\{\rho_\delta,r_0\}$,
\begin{equation*}
\begin{split}
\norm{\nabla u}_{L^p(B_r(z))}&\leq \left(\norm{\nabla u - (\nabla u)_{z,r}}_{L^p(B_r(z))}+\norm{(\nabla u)_{z,r}}_{L^p(B_r(z))}\right)\leq \tilde{C}_\delta r^{-\delta}|B_r|^{\frac{1}{p}},
\end{split}
\end{equation*}
with $\tilde{C}_\delta = C_0+C_\delta$. This finishes the proof.
\end{proof}

Now we are ready to prove the first assertion of Theorem \ref{theorem:main}.

\begin{proof}[Proof of Theorem \ref{theorem:main} (i)] For the sake of simplicity, let us assume $u(0)=0$. Recall the transformation $\phi$ from Definition \ref{definition:phi} and $v=\phi(u)$. Also recall from Remark \ref{remark:phi} that $v$ is a weak solution to   
\begin{equation*}
\Delta v = \ddiv(\vv{f}(x)+\sigma(x)\nabla u),
\end{equation*}
where $\sigma$ is given by \eqref{eq:sigma}, and more importantly, satisfies \eqref{eq:sigma-Ca}.
By \eqref{eq:cont} and \eqref{eq:control-grad u-W1p}, we can choose a positive $r_0$, depending only on $n$, $\lambda$, $\omega_0$ and $\alpha$, such that 
\begin{equation*}
\left(\fint_{B_r}|\sigma\nabla u|^{n+1}\right)^{\frac{1}{n+1}}\leq\eta r^{\frac{\alpha}{2}},
\end{equation*}
for any $0<r\leq r_0$, where $\eta$ is the constant chosen as in Lemma \ref{lemma:C1a}. From this observation, we conclude that $v$ is $C^{1,\frac{\alpha}{2}}$ at the origin.\footnote{Indeed $v$ is $C^{1,\alpha'}$ at the origin for any $\alpha'<\alpha$. Here we chose the exponent to be $\frac{\alpha}{2}$ just to make sure that our estimate is universal, since otherwise the estimate would depend on $\alpha-\alpha'$.} In particular, we have $\sup_{B_r}|v|\leq C_0r$ for any $0<r\leq r_1$, where $r_1$ and $C_0$ are positive constants depending only on $n$, $\lambda$, $\omega_0$ and $\alpha$. Hence, $\sup_{B_r}|u|\leq C_1r$ with $C_1=C_0/\lambda$, for any $0<r\leq r_1$. 

The argument above can be applied for any $z\in\Gamma(u)\cap B_{1/2}$, from which we obtain $\sup_{B_r(z)}|u|\leq C_1r$ for $0<r\leq r_1$, which amounts to the fact that $u$ is Lipschitz continuous across the free boundary. The rest of the proof follows easily from the standard regularity theory, since $u$ solves smooth PDEs on both $\{x\in B_1:u(x)>0\}$ and $\{x\in B_1:u(x)<0\}$. To be more specific, let $x^0\in B_{1/2}\cap \Gamma(u)$ and $r=\dist(x^0,\Gamma(u))$, and define $u_r(y) = u(ry+x^0)/r$. Then $|u_r|\leq C_1$ in $B_1$ and $\ddiv_y(a_r(y)\nabla_y \tilde{u}) = \ddiv_y(\vv{f}_r(y))$ in $B_1$, where $a_r(y) = a_\pm (ry+x^0)$ (depending on whether $u(x^0)$ is positive or negative) and $\vv{f}_r(y) = \vv{f}(ry+x^0)$. Since $a_r$ and $\vv{f}_r$ still satisfy the structure condition \eqref{eq:cont}, we have $u_r\in C^{1,\alpha}(\bar{B}_{1/2})$ with $\norm{u_r}_{C^{1,\alpha}(\bar{B}_{1/2})}\leq C_2$, where $C_2>0$ depends only on $n$, $\lambda$, $\omega_0$ and $\alpha$. In particular, $|\nabla_x u(x^0)| = |\nabla_y u_r (0)|\leq C_2$, which finishes the proof. 
\end{proof}

%%%%%%%%%%%%%%%%%%%%%%%
% Subsection: Analysis on Nondegenerate Points
%%%%%%%%%%%%%%%%%%%%%%%

\subsection{Analysis on Nondegenerate Points}\label{subsection:C1a-fb}

\begin{lemma}\label{lemma:v-C1a} Suppose that $z\in\Gamma(u)$. Then $\phi_z(u)$ is $C^{1,\alpha}$ at $z$; that is, there is a unique polynomial $P_z$ of degree $1$ such that 
\begin{equation}\label{eq:v-C1a}
|\phi_z(u)(x)-P_z(x)|\leq M_1|x-z|^{1+\alpha},
\end{equation}
for any $x\in B_{r_0}(z)$, and 
\begin{equation}\label{eq:v-C1a-grad}
|\nabla P_z| \leq N_1,
\end{equation}
where $r_0$, $M_1$ and $N_1$ are positive constants depending only on $n$, $\lambda$, $\omega_0$ and $\alpha$. 
\end{lemma}

\begin{proof}  It suffices to prove the existence of $P_z$ satisfying the above properties, since then 
  the first inequality in \eqref{eq:v-C1a}  implies the uniqueness.

For the sake of simplicity, let us assume that $z=0$. As we have already observed in Remark \ref{remark:phi}, $\phi(u)$ is a weak solution of \eqref{eq:phi}. From Theorem \ref{theorem:main} (i), we know that $\norm{\nabla u}_{L^\infty(B_1)}\leq L$, which in turn implies the existence of some  $r_0<1$ such that  
\begin{equation*}
\norm{\vv{f}-\vv{f}(0)}_{L^\infty(B_{r_0})}+\norm{\sigma\nabla u}_{L^\infty(B_{r_0})}\leq \eta,
\end{equation*}
where $\eta$ is chosen as in Lemma \ref{lemma:C1a}. Applying Lemma \ref{lemma:C1a} with $\vv{f}$ replaced by $\vv{f}+\sigma\nabla u$, we arrive at the desired conclusion. 
\end{proof}

\begin{notation}\label{notation:pz} Throughout this subsection, we will denote by $P_z$ ($z\in\Gamma(u)$) the approximating polynomial of $\phi_z(u)$ chosen as in Lemma \ref{lemma:v-C1a}, unless stated otherwise.
\end{notation}

Recall from Definition \ref{definition:fb} (ii) that $\cN(u)$ is the nondegenerate part of the nodal set $\Gamma(u)$. It is easy to see that $|\nabla P_z|>0$ for any $z\in\cN(u)$, which enables us to represent unit normals on $\cN(u)$ in terms of $\nabla P_z$ as follows.

\begin{definition}\label{definition:normal} Define a mapping $\nu:\cN(u)\ra \sS^{n-1}$ by 
\begin{equation*}
\nu(z) = \frac{\nabla P_z}{|\nabla P_z|}.
\end{equation*}
\end{definition}

The following lemma justifies why $\nu$ can be understood as the normal mapping of $\Gamma(u)$ towards $\{u>0\}$.

\begin{lemma}\label{lemma:trap-C1a} Suppose that $z\in\cN(u)$. Then $\Gamma(u)$ is locally trapped in between two $C^{1,\alpha}$ graphs touching at $z$. More exactly, one has
\begin{equation}\label{eq:trap-C1a}
B_{r_0}(z)\cap\Gamma(u)\subset\left\{x:|(x-z)\cdot\nu(z)|\leq \frac{C}{\delta}|x-z|^{1+\alpha}\right\},
\end{equation}
where $r_0$ is the same constant chosen in \eqref{eq:v-C1a}, $C$ is a positive constant depending only on $n$, $\lambda$, $\omega_0$ and $\alpha$, and 
\begin{equation*}
\delta = \limsup_{x\ra z} \frac{|u(x)|}{|x-z|}.
\end{equation*}
\end{lemma}

\begin{proof} It is important to notice that $\Gamma(u) = \Gamma(\phi_z(u))$. Due to the choice of $z$, we have $|\nabla P_z|\geq \lambda\delta$. The rest of the proof follows easily from \eqref{eq:v-C1a}. 
\end{proof}

It is noteworthy that the choice of the radius $r_0$ in both \eqref{eq:v-C1a} and \eqref{eq:trap-C1a} has been made uniform over free boundary points, at least in $B_{1/2}$. An immediate consequence is that $\cN(u)$ is relatively open in $\Gamma(u)$, as stated in the following lemma.

\begin{lemma}\label{lemma:rel open} Suppose that $z\in\cN(u)$ and 
\begin{equation*}
\delta = \limsup_{x\ra z} \frac{|u(x)|}{|x-z|}.
\end{equation*}
Then there are positive constants $c_1$ and $c_2$ depending only on $n$, $\lambda$, $\omega_0$ and $\alpha$ such that $\Gamma(u)\cap B_{r_z}(z)\subset\cN(u)$ and that for any $\xi\in\Gamma(u)\cap B_{r_z}(z)$, 
\begin{equation*}
\limsup_{x\ra \xi} \frac{|u(x)|}{|x-\xi|} \geq c_2\delta, 
\end{equation*}
where $r_z>0$ depends only on $n$, $\lambda$, $\omega_0$, $\alpha$ and $z$. 
\end{lemma}

\begin{proof} Translating the solution to the origin if necessary, it suffices to prove the statement for $z=0$. 
The nondegeneracy constant $\delta$ implies that $\sup_{B_r}u\geq 2^{-1}\delta r$ for all $0<r<\rho_0$, for some $\rho_0>0$. Fix such a radius $r$ and choose any $z\in (B_r\setminus\{0\})\cap \Gamma(u)$. Put $\rho=|z|$. Then $\sup_{B_{2\rho}(z)}u\geq \sup_{B_\rho}u\geq \delta \rho$, from which we are able to select $x^0\in B_{2\rho}(z)$ such that $u(x^0)\geq \delta \rho$. 
By definition, 
\begin{equation*}
\phi_z(u)(x^0) = a_+(z)u(x^0) \geq \lambda \delta \rho.
\end{equation*}
By means of \eqref{eq:v-C1a} and $|x^0-z|\leq 2\rho$, we deduce that
\begin{equation*}
|P_z(x^0)| \geq |\phi_z(u)(x^0)| - |\phi_z(u)(x^0) - P_z(x^0)| \geq \rho(\lambda \delta - 2^{1+\alpha}C_1\rho^\alpha),
\end{equation*}
while 
\begin{equation*}
|P_z(x^0)| = |\nabla P_z\cdot(x^0-z)| \leq 2\rho|\nabla P_z|.
\end{equation*}
Combining the above inequalities, we arrive at 
\begin{equation*}
|\nabla P_z| \geq \frac{\lambda\delta}{4},
\end{equation*}
provided that $\rho$ is restricted by $4C_1(2\rho)^\alpha \leq \lambda\delta$. 
Write $\nu_z= \nabla P_z/|\nabla P_z|$ and $x_r=z + r\nu_z$ for $r>0$. We deduce from \eqref{eq:v-C1a} that
\begin{equation*}
\phi_z(u)(x_r) \geq |\nabla P_z|r - C_1r^{1+\alpha} \geq \left(\frac{\lambda\delta}{4}-C_1r^\alpha\right)r.
\end{equation*}
Therefore, for $r^\alpha \leq \frac{\lambda\delta}{8 C_1}$  there holds
\begin{equation*}
\sup_{B_r(z)} u \geq u(x_r) \geq \lambda \phi_z(u)(x_r) \geq \frac{\lambda^2\delta}{8}r .
\end{equation*}
 Thus, the lemma is proved with $c_1\leq\frac{\lambda}{2^{2+\alpha}C_1}$ and $c_2\leq \frac{\lambda^2}{8}$.
\end{proof}

From the previous lemmas, the  $C^{1,\alpha}$ regularity of the free boundary near a nondegenerate point follows in a standard way. For the reader's convenience  we will prove this 
in detail.

\begin{proof}[Proof of Theorem \ref{theorem:main} (ii)] For brevity, let us assume $z=0\in\Gamma(u)$. Since 
\begin{equation*}
\limsup_{x\ra 0}\frac{|u(x)|}{|x|} >0,
\end{equation*}
we may suppose without losing any generality that 
\begin{equation*}
\limsup_{x\ra 0}\frac{|u(x)|}{|x|} = 1.
\end{equation*}
The general case will follow from a normalization argument. 

Choose any $z\in B_{r_0}\cap \Gamma(u)$ other than the origin, where $r_0$ is selected as in Lemma \ref{lemma:rel open}. Write $\rho=|z|$ and consider any $x\in\p B_\rho(z)\cap \Gamma(u)$. Writing $x=z+\rho e$ with $e$ a unit vector, we may apply Lemma \ref{lemma:trap-C1a} to the origin from which we obtain  
\begin{equation}\label{eq:e-nu0}
\rho|e\cdot\nu(0)|\leq |(\rho e + z)\cdot\nu(0)|+|z\cdot\nu(0)|\leq 3C\rho^{1+\alpha},
\end{equation}
since $|\rho e +z|\leq 2\rho$. 

Now let us choose $e^i$ for $1\leq i\leq n-1$ such that $z+\rho e^i\in\Gamma(u)$ and that $\{e^1,\cdots,e^{n-1},\nu(z)\}$ becomes a basis for $\R^n$. Such a basis always exists due to the uniform nondegeneracy (Lemma \ref{lemma:rel open}) and $C^{1,\alpha}$ approximation result (Lemma \ref{lemma:v-C1a}). Since
\begin{equation*}
\nu(0) = \sum_{i=1}^{n-1}(\nu(0)\cdot e^i)e^i + (\nu(0)\cdot \nu(z))\nu(z),
\end{equation*}
taking the inner product with $\nu(0)$ on both sides, we observe that
\begin{equation*}
1 = \sum_{i=1}^{n-1}(\nu(0)\cdot e^i)^2 + (\nu(0)\cdot \nu(z))^2,
\end{equation*}
from which it follows that
\begin{equation*}
|\nu(z)-\nu(0)|^2 = 2\left\{1-\left(1-\sum_{i=1}^{n-1}(\nu(0)\cdot e^i)^2\right)^{\frac{1}{2}}\right\}\leq 2\sum_{i=1}^{n-1}(\nu(0)\cdot e^i)^2.
\end{equation*}
By \eqref{eq:e-nu0}, we conclude that 
\begin{equation*}
|\nu(z)-\nu(0)|\leq \bar{C}\rho^\alpha,
\end{equation*}
where $\bar{C}\geq 3C\sqrt{2(n-1)}$, finishing the proof.
\end{proof}

\begin{corollary} Let the coefficient of \eqref{eq:main} be given by 
\begin{equation}\label{eq:general}
A(x,u)=a_+(x,u)H(s)+a_-(x,u)(1-H(u)),
\end{equation}
where $a_+$ and $a_-$ satisfy \eqref{eq:ellip} and \eqref{eq:cont} in both $x$ and $u$ variables. Then the same conclusion of Theorem \ref{theorem:main} holds.
\end{corollary}

\begin{proof} Let $u$ be a weak solution of \eqref{eq:main} with $|u|\leq 1$ in $B_1$, where the coefficient, $A(x,u)$, is given by \eqref{eq:general}. By the interior H\"{o}lder regularity theory, we have $u\in C^\gamma(\bar{B}_{7/8})$, for some $0<\gamma<1$ depending only on $n$ and $\lambda$. Setting $b_\pm(x)=a_\pm(x,u(x))$, we observe that $u$ solves the broken PDE, $\ddiv(B(x,u)\nabla u)=0$ in $B_{7/8}$, with $B(x,u)=b_+(x)H(u)+b_-(x)(1-H(u))$. However, $b_\pm\in C^{\alpha\gamma}(\bar{B}_{7/8})$ and $b_\pm$ also satisfies the ellipticity condition \eqref{eq:ellip} of $a_\pm$. Hence, we may apply Theorem \ref{theorem:main} (i) to $u$ and deduce that $u\in C^{0,1}(\bar{B}_{3/4})$. Therefore, the coefficient $b_+$ and $b_-$ become $\alpha$-H\"{o}lder continuous in $\bar{B}_{3/4}$, so now from Theorem \ref{theorem:main} (ii) we get $C^{1,\alpha}$ regularity of the free boundary of $u$ near a nondegenerate point. We leave out the details to the reader. 
\end{proof}

%%%%%%%%%%%%%%%%%%%%%%%
% Subsection: Structure Theorem for Degenerate Points
%%%%%%%%%%%%%%%%%%%%%%%

\subsection{Structure Theorem for Degenerate Points}\label{subsection:deg}

From now on, we assume that $\vv{f}\equiv 0$ in \eqref{eq:main} and begin with classification of vanishing orders of degenerate points, namely the points contained in $\cS(u)$. Notice that $\cS(u)$ consists of all vanishing points with finite order which is greater than $1$. A question was raised in \cite{KLS}, asking whether finitely vanishing points are always of integer orders. Here we give an affirmative answer to this question.

\begin{lemma}\label{lemma:u-d} If $z\in\Gamma(u)$ satisfies 
\begin{equation*}
\limsup_{x\ra z}\frac{|u(x)|}{|x-z|^{d-1}}=0,
\end{equation*}
for some integer $d\geq 2$, then there exist positive constants $C_d$ and $r_d$, depending only on $n$, $\lambda$, $\omega_0$, $\alpha$ and $d$, such that
\begin{equation}\label{eq:u-d}
|u(x)|\leq C_d|x-z|^d,
\end{equation}
for any $x\in B_{r_d}(z)$.
\end{lemma}

\begin{proof} Without loss of generality we may assume $z=0$. We shall only prove this for the case  $d=2$. The proof for  $d> 2$ follows by a similar reasoning, and is left to the reader. 

Write $v=\phi(u)$. By the assumption, we have $\nabla P=0$, where $P$ is the approximating linear  polynomial of $v$ chosen from Lemma \ref{lemma:v-C1a}. Since $P(0) = 0$, we conclude that $P$ is identically zero. Hence, the inequality \eqref{eq:v-C1a} is reduced as
\begin{equation}\label{eq:v-C1a-deg}
|v(x)|\leq M_1|x|^{1+\alpha},
\end{equation}
for any $x\in B_{r_0}$, with $M_1$ being the constant chosen from Lemma \ref{lemma:v-C1a}.

Let $v_r(x)=\frac{1}{r^{1+\alpha}}v(rx)$. In view of \eqref{eq:phi}, one may easily verify that
\begin{equation}\label{eq:vr}
\Delta v_r=\ddiv(\sigma_r(x)\nabla u_r),
\end{equation}
where $\sigma_r(x)=\frac{1}{r^\alpha}\sigma(rx)$ and $u_r(x)=\frac{u(rx)}{r}$. Using Theorem \ref{theorem:main} (i) and the structure condition \eqref{eq:cont}, we obtain
\begin{equation}\label{eq:sigmar}
\norm{\sigma_r\nabla u_r}_{L^\infty(B_1)}\leq C_1,
\end{equation}
for any $0<r\leq r_0$, where $C_1$ is a constant depending only on $n$, $\lambda$, $\omega_0$ and $\alpha$. Also it follows from \eqref{eq:v-C1a-deg} that 
\begin{equation}\label{eq:vr-Linf}
\norm{v_r}_{L^\infty(B_1)}\leq M_1,
\end{equation}
for any $0<r\leq r_0$. Applying the interior $W^{1,p}$ estimate ($p>n$) on \eqref{eq:vr} and utilizing the estimates \eqref{eq:sigmar} and \eqref{eq:vr-Linf}, we obtain
\begin{equation*}
\left(\fint_{B_{1/2}}|\nabla v_r|^p\right)^{1/p}\leq C_p(M_1+C_1),
\end{equation*}
where $C_p$ is another positive constant depending only on $n$, $\lambda$, $\omega_0$, $\alpha$ and $p$. 
Scaling back, we will have for any $r\leq\frac{r_0}{2}$, 
\begin{equation*}
\left(\fint_{B_r}|\sigma\nabla u|^p\right)^{1/p}\leq\omega_0r^\alpha\left(\fint_{B_r}|\nabla u|^p\right)^{1/p}\leq\frac{\omega_0r^{2\alpha}}{\lambda}\left(\fint_{B_{1/2}}|\nabla v_r|^p\right)^{1/p}\leq\frac{\omega_0\tilde{C}_p}{\lambda}r^{2\alpha},
\end{equation*} 
with $\tilde{C}_p = C_p(M_1+C_1)$. 

Replacing $\alpha$ by an irrational number $\alpha'<\alpha$ if necessary, we may always assume that $\alpha$ is irrational. Now if $2\alpha<1$, we may apply Lemma \ref{lemma:C1a} with $d=0$ and $p=n+1$ such that
\begin{equation*}
|v(x)|\leq K_1|x|^{1+2\alpha}\quad\text{in }B_{r_1},
\end{equation*}
where $r_1=r_0/2$ and $K_1$ is a positive constant depending only on $n$, $\lambda$, $\omega_0$ and $\alpha$. In comparison with \eqref{eq:v-C1a-deg}, we have achieved an improvement of the decay rate of $v$ around the origin. 

Notice that we are able to iterate the above argument so as to obtain
\begin{equation*}
|v(x)|\leq K_{m-1}|x|^{1+m\alpha}\quad\text{in }B_{r_{m-1}},
\end{equation*}
where $m$ is the largest integer so that $m\alpha<1$ and $K_m$ a constant depending only on $n$, $\lambda$, $\omega_0$, $\alpha$ and $m$. Now as $\alpha$ being irrational, we have $(m+1)\alpha>1$. Set $\alpha_m=(m+1)\alpha-1$ and apply the argument above once again. This leads us to
\begin{equation*}
\left(\fint_{B_r}|\sigma\nabla u|^p\right)^{1/p}\leq L_mr^{(m+1)\alpha}=L_mr^{1+\alpha_m},
\end{equation*}
with $L_m$ being a positive constant depending only on $n$, $\lambda$, $\omega_0$, $\alpha$ and $m$. Then we can apply the same lemma with $d=1$, from which we obtain a homogeneous harmonic polynomial $P$ of degree $2$ such that 
\begin{equation}\label{eq:v-C2am-deg}
|v(x)-P(x)|\leq K_m|x|^{2+\alpha_m},
\end{equation}
for any $x\in B_{r_m}$, and 
\begin{equation}\label{eq:v-C2am-deg-hess}
\norm{D^2 P} \leq \tilde{K}_m,
\end{equation}
where $r_m$, $K_m$ and $\tilde{K}_m$ are positive constants depending only on $n$, $\lambda$, $\omega_0$, $\alpha$ and $m$. Notice that $m$ is a parameter determined only by $\alpha$. Our conclusion \eqref{eq:u-d} is now an easy consequence of \eqref{eq:v-C2am-deg} and the triangle inequality, which makes our proof complete.
\end{proof}

\begin{definition}\label{definition:Sd} Given a continuous function $v$ and a positive integer $d\geq 2$, define 
\begin{equation*}
\cS_d(v) = \left\{z\in\Gamma(v): 0<\limsup_{x\ra z}\frac{|v(x)|}{|x-z|^d}<\infty\right\}.
\end{equation*}
\end{definition}

It is an easy consequence from Lemma \ref{lemma:v-Cda-deg} that 
\begin{equation*}
\cS(u) = \bigcup_{d\geq 2}\cS_d(u).
\end{equation*}
Due to Lemma \ref{lemma:u-d}, we are able to proceed with a higher order approximation of $\phi_z(u)$ at any degenerate point $z$ as in the previous subsection. As a result, we derive the following  two lemmas, whose proofs are omitted and left to the reader. 

\begin{lemma}\label{lemma:v-Cda-deg} 
Suppose that 
\begin{equation*}
\limsup_{x\ra z}\frac{|u(x)|}{|x-z|^{d-1}}=0
\end{equation*}
at a point $z\in\cS(u)$ for some integer $d\geq 2$. Then $\phi_z(u)$ is $C^{d,\alpha}$ at $z$; that is, there exists a unique harmonic polynomial $P_z$ such that $x\mapsto P_z(x+z)$ is homogeneous of degree greater than or equal to $d$ and 
\begin{equation}\label{eq:v-Cda-deg}
|\phi_z(u)(x)-P_z(x)|\leq M_d|x-z|^{d+\alpha},
\end{equation}
for any $x\in B_{r_d}(z)$, and 
\begin{equation}\label{eq:v-Cda-deg-Db}
\sum_{|\beta|=d}|D^\beta P_z|\leq N_d,
\end{equation}
where $r_d$, $M_d$ and $N_d$ are constants depending only on $n$, $\lambda$, $\omega_0$, $\alpha$ and $d$. 
\end{lemma}

\begin{lemma}\label{lemma:rel open-deg} Suppose $z\in\cS_d(u)$ and 
\begin{equation*}
0 < \delta =\limsup_{x\ra z}\frac{|u(x)|}{|x-z|^d}.
\end{equation*} 
Then there are positive constants $c_1$ and $c_2$, depending only on $n$, $\lambda$, $\omega_0$, $\alpha$ and $d$, such that $B_{r_z}(z)\cap(\Gamma(u)\setminus\bigcup_{k\leq d-1}\cS_k(u))\subset \cS_d(u)$ for some $0<r_z<(c_1\delta)^{1/\alpha}$ and that 
\begin{equation*}
\limsup_{x\ra\xi}\frac{|u(x)|}{|x-\xi|^d}\geq c_2\delta,
\end{equation*}
for any $\xi\in B_{r_z}(z)\cap (\Gamma(u)\setminus\bigcup_{k\leq d-1}\cS_k(u))$.
\end{lemma}

In view of  Lemma \ref{lemma:v-Cda-deg}, we can  update Notation \ref{notation:pz} as follows. 

\begin{notation}\label{notation:pz-deg} From now on to the end of this section, $P_z$ ($z\in\cS(u)$) will stand for the approximating polynomial of $\phi_z(u)$ chosen from Lemma \ref{lemma:v-Cda-deg}.
\end{notation}

The following lemma is proved in \cite{H} but we provide the proof for the sake of completeness. 

\begin{lemma}\label{lemma:hom poly} Let $P$ be nontrivial homogeneous polynomial of degree $d\geq 1$. Then 
\begin{equation}\label{eq:hom poly-1}
\cS_d(P) = \{z\in \R^n: D^\nu P(z)=0\text{ for any }|\nu|\leq d-1\}.
\end{equation}
Moreover, 
\begin{equation}\label{eq:hom poly-2}
P(x+z) = P(x)\quad (x\in\R^n,z\in\cS_d(P)),
\end{equation}
and furthermore, $\cS_d(P)$ is a linear subspace in $\R^n$. In particular, if we further assume that $P$ is harmonic with degree $d\geq 2$, then $\dim(\cS_d(P))\leq n-2$. 
\end{lemma}

\begin{proof} Identity \eqref{eq:hom poly-1} follows immediately from the homogeneity of $P$. The proof of \eqref{eq:hom poly-2} is also elementary. If $z\in\cS_d(P)$, then $D^\beta P(z) = 0$ for any $|\beta|\leq d-1$, whence for any $x\in\R^n$, 
\begin{equation*}
P(x+z) = \sum_{|\beta|\leq d}\frac{D^\beta P(z)}{\beta!}x^\nu = \sum_{|\beta|=d}\frac{D^\beta P(z)}{\beta!}x^\beta = \sum_{|\beta|=d}\frac{D^\beta P(0)}{\beta!}x^\beta = P(x),
\end{equation*}
which proves \eqref{eq:hom poly-2}.

Now since $P$ is homogenous of degree $d$, for any nonzero $r\in\R$, 
\begin{equation*}
P(x+rz) = r^d P(r^{-1}x+z) = r^d P(r^{-1}x) = P(x)\quad (x\in\R^n),
\end{equation*}
where we have used \eqref{eq:hom poly-2} in the second equality. Again by the Taylor expansion, 
\begin{equation*}
\sum_{|\beta|\leq d-1}\frac{D^\beta P(rz)}{\beta!}x^\beta = 0\quad (x\in\R^n),
\end{equation*}
which implies that 
\begin{equation*}
D^\beta P(rz) = 0\quad\text{for any }|\beta|\leq d-1.
\end{equation*}
Thus, $rz\in \cS_d(P)$ for all $r\in\R$ and $z\in\cS_d(P)$. In a similar way, we can prove that $y+z\in\cS_d(P)$, if $y\in\cS_d(P)$ and $z\in\cS_d(P)$, which proves that $\cS_d(P)$ is a linear subspace of $\R^n$. 

For the last statement of this lemma, we observe that $P$ is a polynomial of $n-\dim(\cS_d(P))$ variables. Now if we assume that $\dim(\cS_d(P))=n-1$, then $P$ becomes a nontrivial monomial of degree $d\geq 2$, which cannot be a harmonic function. Thus, $\dim(\cS_d(P))\leq n-2$, which proves the last part of the lemma.
\end{proof}

We are now ready to prove our second main theorem. Fix an integer $d\geq 2$. We are going to prove the statement for $\cS_d(u)$ instead of $\cS(u)$. Then the proof follows from the fact that $\cS(u)$ is a countable union of $\cS_d(u)$ for all $d\geq 2$. 

As considered in \cite{H}, let us further decompose $\cS_d(u)$ by 
\begin{equation*}
\cS_d^j(u) = \{z\in\cS_d(u): \dim(\cS_d(P_z)) = j\},\quad j=0,1,\cdots,n-2,
\end{equation*}
where $P_z$ is the approximating harmonic polynomial (of $\phi_z(u)$) chosen in Lemma \ref{lemma:v-Cda-deg}. By Lemma \ref{lemma:hom poly}, we know that
\begin{equation*}
\cS_d(u) = \displaystyle\bigcup_{j=0}^{n-2}\cS_d^j(u).
\end{equation*}

In the subsequent two lemmas, we show that for each $z\in\cS_d^j(u)$, there exists $r_z>0$ such that $B_{r_z}(z)\cap \cS_d^j(u)$ is a $j$-dimensional $C^1$ graph. The essential idea of these lemmas follows from \cite[Lemma 2.2]{H}. However, we need to be very careful when constructing blowup limits, since here our transformation $\phi_z(u)$ is $C^{d,\alpha}$ only at $z\in\cS_d^j(u)$ but not necessarily at other points in $\cS_d^j(u)$. 

First we prove that there is a lower-dimensional tangent subspace at each point in $\cS_d^j(u)$. 

\begin{lemma}\label{lemma:main-deg-tangent} If $0\in\cS_d^j(u)$, then $\cS_d(P_0)$ is a tangent linear subspace to $\cS_d^j(u)$ at $0$, where $P_0$ is the approximating harmonic polynomial (of $\phi(u)$) chosen in Lemma \ref{lemma:v-Cda-deg}.
\end{lemma}

\begin{proof}
Let $z^k\in\cS_d^j(u)$ such that $z^k\ra 0$. Put $\rho_k=|z^k|$ and $\xi^k=\frac{z^k}{\rho_k}\in\p B_1$. The claim will be proved if one shows that any limit of $\{\xi^k\}_{k=1}^\infty$ belongs to $\cS_d(P_0)$.

Take a limit $\xi^0$ of $\{\xi^k\}_{k=1}^\infty$ along a possible subsequence and consider the scaled version of $v_0=\phi_0(u)$,
\begin{equation*}
v_k(x)=\frac{v_0(\rho_k x)}{\rho_k^d}.
\end{equation*}
It follows from \eqref{eq:v-Cda-deg} that 
\begin{equation}\label{eq:vk-deg}
|v_k(x)-P_0(x)|\leq M_d|x|^{d+\alpha}\rho_k^\alpha\quad\text{in }B_2,
\end{equation}  
provided that $\rho_k\leq\frac{r_d}{2}$ for all $k$.
On the other hand, $z^k\in\cS_d(u)$ implies $\xi^k\in\cS_d(v_k)$, since
\begin{equation}\label{eq:xik}
\limsup_{x\ra\xi^k}\frac{|v_k(x)|}{|x-\xi^k|^d}=\limsup_{z\ra z^k}\frac{|v_0(z)|}{|z-z^k|^d}\leq\frac{1}{\lambda}\limsup_{z\ra z^k}\frac{|u(z)|}{|z-z^k|^d}\leq C_d.
\end{equation}
Due to \eqref{eq:u-d}, the constant, $C_d$, on the rightmost side of \eqref{eq:xik}, can be chosen independently of $z^k$.

Now let $x\in B_2$ such that $|x-\xi^0|>0$. Since $\xi^k\ra \xi^0$, we can always find $k$ satisfying $\rho_k\leq |x-\xi^0|^{d/\alpha}$ and $|x-\xi^k|\leq\frac{3}{2}|x-\xi^0|$. Then it follows from \eqref{eq:vk-deg} and \eqref{eq:xik} that
\begin{equation*}
\frac{|P_0(x)|}{|x-\xi^0|^d}\leq\frac{|v_k(x)-P(x)|}{|x-\xi^0|^d}+\frac{|v_k(x)|}{|x-\xi^k|^d}\frac{|x-\xi^k|^d}{|x-\xi^0|^d}\leq 2^{d+\alpha}M_d+\frac{3}{2}C_d\leq \bar{C}_d,
\end{equation*}
where $\bar{C}_d>0$ depends only on $n$, $\lambda$, $\omega_0$, $\alpha$ and $d$. It shows that $\xi^0\in\cS_d(P_0)$, which proves the claim.
\end{proof}

Next we prove that the tangent linear subspace in Lemma \ref{lemma:main-deg-tangent} can be chosen uniformly in a neighborhood of a point in $\cS_d^j(u)$, proving the continuous differentiability of this set.

\begin{lemma}\label{lemma:main-deg-C1} If $0\in \cS_d(u)$ and there is a sequence $\{z^k\}_{k=1}^\infty\subset\cS_d(u)$ converging to $0$, then $P_{z^k}\ra P_0$ in $C_{loc}^d(\R^n)$. 
\end{lemma}

\begin{proof} First let us observe that 
\begin{equation*}
\begin{split}
|v_{z^k}(x) - v_0(x)|&\leq |a_+(z^k)-a_+(0)|u^+(x) + |a_-(z^k)-a_-(0)|u^-(x)\\
&\leq \omega_0\rho_k^\alpha|u(x)|,
\end{split}
\end{equation*}
where $\rho_k=|z^k|$. Hence, it follows from Lemma \ref{lemma:v-Cda-deg} that
\begin{equation}\label{eq:pzk-p0}
\begin{split}
|P_{z^k}(x)-P_0(x)| &\leq |P_{z^k}(x) - v_{z^k}(x)| + |v_{z^k}(x) - v_0(x)| + |v_0(x) - P_0(x)|\\
&\leq M_d(|x-z^k|^{d+\alpha}+|x|^{d+\alpha})+\omega_0\rho_k^\alpha|u(x)|,
\end{split}
\end{equation}
provided that $x\in B_{r_d}\cap B_{r_d}(z^k)$. 

Set  $\xi^k=z^k/\rho_k\in \p B_1$ and let  $\xi^0$  be a limit of $\{\xi^k\}_{k=1}^\infty$ along a  subsequence. Define further 
\begin{equation*}
v_k(x)=\frac{v_{z^k}(z^k+\rho_kx)}{\rho_k^d},
\end{equation*}
where $\rho_k=|z^k|$, and 
\begin{equation*}
P_k(x)=\frac{P_{z^k}(z^k+\rho_kx)}{\rho_k^d}=P_{z^k}(\xi^k+x).
\end{equation*}
By Lemma \ref{lemma:hom poly}, we have  
\begin{equation}\label{eq:pk-pzk}
P_k = P_{z^k}\quad\text{in }\R^n.
\end{equation}

In view of the equation \eqref{eq:phi}, $v_k$ solves 
\begin{equation}\label{eq:vk-eq}
\Delta v_k=\ddiv(\sigma_k(x)\nabla u_k)\quad\text{in }B_{1/\rho_k},
\end{equation}
where
\begin{equation}\label{eq:sigmak-uk}
\norm{\sigma_k\nabla u_k}_{L^p(B_R)}\leq C_d\rho_k^\alpha R^{d-1+\alpha\frac{n}{p}},
\end{equation}
for any $0<\rho_k R<1$, with $C_d>0$ depending only on $n$, $\lambda$, $\omega_0$, $\alpha$ and $k$. On the other hand, Lemma \ref{lemma:rel open-deg} yields some positive constants $\delta_0$ and $r_0$, depending only on $n$, $\lambda$, $\omega_0$, $\alpha$ and $d$, such that for all sufficiently large $k$, there holds
\begin{equation}\label{eq:vk-delta0}
\sup_{B_r}|v_k|\geq\delta_0r^d,
\end{equation}
for any $0<\rho_k r<r_0$. Moreover we may deduce from \eqref{eq:v-Cda-deg} and \eqref{eq:v-Cda-deg-Db} that
\begin{equation}\label{eq:vk-Pk}
|v_k(x)-P_k(x)|\leq M_d\rho_k^\alpha|x|^{d+\alpha}\quad\text{and}\quad |v_k(x)|\leq C_d|x|^d,
\end{equation}
for any $x\in B_{1/\rho_k}$. Applying the interior $W^{1,p}$ estimate on \eqref{eq:vk-eq} and using \eqref{eq:sigmak-uk} together with \eqref{eq:v-Cda-deg} and \eqref{eq:v-Cda-deg-Db}, we obtain a locally uniform bound on $L^p$ norm of $\nabla v_k$. By compactness, there is a subsequence of $\{v_k\}_{k=1}^\infty$ which converges to a function $w_0$ uniformly in $C_{loc}^{1-p/n}(\R^n)$ and weakly in $W_{loc}^{1,p}(B_1)$. In particular, the fact that $v_k$ solves \eqref{eq:vk-eq} in weak sense implies that
\begin{equation*}
\Delta w_0 = 0 \quad\text{in }\R^n, 
\end{equation*}
while \eqref{eq:vk-delta0} and \eqref{eq:vk-Pk} ensure that
\begin{equation*}
\delta_0r^d\leq \sup_{B_r}|w_0|\leq C_dr^d\quad (r>0).
\end{equation*}
As $w_0$ being a harmonic function satisfying $d$-th order growth condition, $w$ must be a homogeneous harmonic polynomial of degree $d$. Moreover, the first inequality in \eqref{eq:vk-Pk} implies that $\{P_k\}_{k=1}^\infty$ also converges to $w_0$ locally uniformly in $\R^n$. Since $w_0$ and $P_k$ ($k=1,2,\cdots$) are all homogeneous polynomial of degree $d$, the uniform convergence of $\{P_k\}_{k=1}^\infty$ to $w_0$ on $\p B_1$ implies that the leading coefficients of $P_k$ converges to those of $w_0$. However, since the leading coefficients are their $d$-th order derivatives up to constant factors, we obtain that $P_k\ra w_0$ in $C_{loc}^d(\R^n)$ as $k\ra\infty$. 

In view of \eqref{eq:pk-pzk}, we are only left with proving that $w_0=P_0$. Passing limit $k\ra\infty$ in \eqref{eq:pzk-p0}, we obtain 
\begin{equation*}
|w_0(x) - P_0(x)| \leq 2M_d|x|^{d+\alpha}\quad\text{in }B_{r_d}.
\end{equation*}
If $w_0\neq P_0$, then by the homogeneity of these two functions, there must exist $x^0\in\p B_1$ such that $|w_0(x^0) - P_0(x^0)|>\e$ for some $\e>0$. However, taking $r>0$ small enough such that $4M_dr^\alpha\leq\e$ and $r\leq r_d$, we arrive at
\begin{equation*}
\e < |w_0(x^0) - P_0(x^0)| = r^{-d}|w_0(rx^0) - P_0(rx^0)| \leq 2M_dr^\alpha \leq \frac{\e}{2},
\end{equation*}
a contradiction. This observation implies  $w_0\equiv P_0$, finishing the proof of Lemma \ref{lemma:main-deg-C1}.
\end{proof}

We are ready to prove the structure theorem of the nodal set for broken PDEs. 

\begin{proof}[Proof of Theorem \ref{theorem:main-deg} (i)]

Lemma \ref{lemma:main-deg-tangent} implies that for each $z\in\cS_d^j(u)$ and $\e>0$, there exists $r_{z,\e}>0$ such that
\begin{equation*}
B_{r_{z,\e}}(z)\cap \cS_d^j(u)\subset \{x\in\R^n: \dist(x,\cS_d(P_z))<\e|x-z|\}.
\end{equation*}
By Lemma \ref{lemma:main-deg-C1}, the choice of $r_{z,\e}$ can be made uniform in a small neighborhood of $z$ such that for any $y\in B_{r_{z,\e}}(z)\cap \cS_d^j(u)$,
\begin{equation*}
B_{r_{z,\e}}(y)\cap \cS_d^j(u) \subset \{x\in\R^n:\dist(x,\cS_d(P_y))<\e|x-y|\}.
\end{equation*}
This amounts to the fact that for each $z\in\cS_d^j(u)$, there exists $r_z>0$ such that $B_{r_z}\cap \cS_d^j(u)$ is a $j$-dimensional $C^1$ graph. The rest of the proof follows from a simple covering argument, which we leave out to the reader.
\end{proof}

%%%%%%%%%%%%%%%%%%%%%%%
% Subsection: Hausdorff Measure of Nodal Sets
%%%%%%%%%%%%%%%%%%%%%%%

\subsection{Hausdorff Measure of Nodal Sets}\label{subsection:hausdorff}

We are now left with proving the second assertion of Theorem \ref{theorem:main-deg}. Let us first observe a general fact associated with the Lebesgue measure of positive and negative set of $u$, when we have a control on its frequency.\footnote{That is, a control on $\dfrac{\int_{B_1}|\nabla u|^2}{\int_{\p B_1} u^2}$.}

\begin{lemma}\label{lemma:leb} Let $u$ be a nontrivial weak solution to \eqref{eq:main} in $B_1$. Then there exist a positive constant $\mu$, which depends only on $n$, $\lambda$ and $N$, such that
\begin{equation}\label{eq:leb}
\min\{\cL^n(\{u>0\}\cap B_1) , \cL^n(\{u<0\}\cap B_1)\} \geq \mu,
\end{equation}
where 
\begin{equation*}
N := \frac{\int_{B_1}|\nabla u|^2}{\int_{\p B_1} u^2}.
\end{equation*}
\end{lemma}

\begin{proof} Suppose towards a contradiction that the statement is false. By normalization, we may without loss of generality choose a sequence $\{u_j\}_{j=1}^\infty$ of nontrivial weak solutions to \eqref{eq:main} in $B_1$ such that 
\begin{equation}\label{eq:uj-1}
\int_{B_1}|\nabla u_j|^2 = N\quad\text{and}\quad \int_{\p B_1} u_j^2=1,
\end{equation}
whereas
\begin{equation}\label{eq:uj-2}
\max\{\cL^n(\{u_j>0\}\cap B_1),\cL^n(\{u_j<0\}\cap B_1)\}\leq \frac{1}{j}.
\end{equation}

It follows from \eqref{eq:uj-1} that 
\begin{equation*}
\int_{B_1}u_j^2 \leq C_N,
\end{equation*}
where $C_N$ is a positive constant depending only on $n$ and $N$. Hence, the De Giorgi theory yields a locally uniform $\gamma$-H\"{o}lder estimate for $\{u_j\}_{j=1}^\infty$, where $0<\gamma<1$ depends only on $n$ and $\lambda$. Therefore we are able to extract a subsequence $\{u_{j_k}\}_{k=1}^\infty$ of $\{u_j\}_{j=1}^\infty$ and a function $u_0\in C_{loc}^\gamma(B_1)\cap W^{1,2}(B_1)$ such that $u_{j_k}\ra u_0$ in $L^2(\p B_1)$, weakly in $W^{1,2}(B_1)$ and locally uniformly in $B_1$.

From the weak convergence of $\{u_{j_k}\}_{k=1}^\infty$ in $W^{1,2}(B_1)$ we deduce that $u_0$ also solves \eqref{eq:main} in $B_1$. On the other hand, the locally uniform convergence implies that given $\e>0$, $\eta>0$ and $D\Subset B_1$, there exists $k_0= k_0(\e,D)$ such that $\{u_0 > 2\e\} \cap D \subset \{u_{j_k}> \e\}\cap D$ for all $k\geq k_0$. Since $\{u_{j_k}>\e\}\cap D\subset \{u_{j_k}>0\}\cap B_1$, 
\begin{equation*}
\cL^n(\{u_0>2\e\}\cap D)\leq \cL^n(\{u_{j_k}>0\}\cap B_1) \leq \frac{1}{j_k}\quad (k\geq k_0).
\end{equation*}
Letting $k\ra \infty$ and then $\e\ra 0$, we obtain from a covering argument that 
\begin{equation*}
\cL^n(\{u_0>0\}\cap B_1) = 0.
\end{equation*}
Similarly, one may show that 
\begin{equation*}
\cL^n(\{u_0<0\}\cap B_1) = 0.
\end{equation*}
By means of the continuity of $u_0$, we deduce that $u_0 = 0$ in $B_1$. In particular, $u_0 = 0$ on $\p B_1$ in the trace sense. However, from \eqref{eq:uj-1} and the strong convergence of $\{u_{j_k}\}_{k=1}^\infty$ in $L^2(\p B_1)$, we arrive at
\begin{equation*}
\int_{\p B_1}u_0^2 = 1,
\end{equation*}
a contradiction.
\end{proof}

An easy corollary to the above lemma is that we may apply the Poincar\'{e} inequality directly to $u$, without subtracting its average, whenever we have a control on the frequency of $u$. This fact will be used in the proof of Theorem \ref{theorem:main-deg} (ii). As an independent remark, we would like to point out that the proof of Lemma \ref{lemma:leb} does not involve any regularity on the coefficients of \eqref{eq:main}. 

Let us define a function $w$ by 
\begin{equation*}
w(x) = a_+(x) u^+(x) - a_-(x) u^-(x).
\end{equation*}
One should notice that the zero level surface of $w$ coincides with that of $u$. It is also noteworthy that under the assumption that $a_+,a_- \in W^{1,2}(B_1)$, $w$ becomes a weak solution to 
\begin{equation}\label{eq:w}
\Delta w = \ddiv(u^+(x)\nabla a_+ - u^-(x)\nabla a_-)\quad\text{in }B_1.
\end{equation}

\begin{proof}[Proof of Theorem \ref{theorem:main-deg} (ii)] For notational convenience, let us write 
\begin{equation*}
M = \max\left\{\norm{a_+}_{C^{1,1}(B_1)},\norm{a_-}_{C^{1,1}(B_1)}\right\}\quad\text{and}\quad N = \frac{\int_{B_1}|\nabla u|^2}{\int_{\p B_1} u^2}.
\end{equation*}
By the definition of $w$, we have $u^\pm = w^\pm/a_\pm$, from which we may rephrase \eqref{eq:w} as
\begin{equation}\label{eq:w-1}
\Delta w = \vv{b}(x)\nabla w + c(x)w\quad \text{in }B_1,
\end{equation}
where
\begin{align*}
\vv{b}(x) &= \frac{1}{a_+(x)}\nabla a_+(x)H(u(x))  + \frac{1}{a_-(x)}\nabla a_-(x)(1-H(u(x))),\\
c(x) &=  \ddiv\left(\frac{1}{a_+(x)}\nabla a_+(x)\right)H(u(x)) + \ddiv\left(\frac{1}{a_-(x)}\nabla a_-(x)\right)(1-H(u(x))).
\end{align*}
Note that the $C^{1,1}$ regularity of $a_+$ and $a_-$ assures the boundedness of $\vv{b}$ and $c$ in $B_1$ with 
\begin{equation*}
\norm{\vv{b}}_{L^\infty(B_1)}+\norm{c}_{L^\infty(B_1)}\leq C_1M,
\end{equation*}
where $C_1>0$ depends only on $n$ and $\lambda$. 

In addition, observe that the frequency of $w$ is comparable to that of $u$. By the definition of $w$, we have 
\begin{equation*}
\int_{\p B_1}w^2 \geq \lambda^2 \int_{\p B_1} u^2.
\end{equation*}
On the other hand, from Lemma \ref{lemma:leb} we may apply the Poincar\'{e} inequality to $u^+$ and respectively $u^-$ to obtain that 
\begin{equation*}
\int_{B_1}u^2\leq \kappa\int_{B_1} |\nabla u|^2,
\end{equation*}
where $\kappa$ is a positive constant depending only on $n$, $\lambda$ and $N$. Combining this inequality with the fact that  
\begin{equation*}
\nabla w^\pm(x) = \frac{1}{a_\pm(x)}\nabla u^\pm(x) - \frac{1}{(a_\pm(x))^2}u^\pm(x)\nabla a_\pm(x),
\end{equation*}
we arrive at
\begin{equation*}
\int_{B_1} |\nabla w|^2 \leq 2M\int_{B_1}(|\nabla u|^2 + u^2)\leq (1+\kappa)M\int_{B_1}|\nabla u|^2.
\end{equation*}
Thus, the frequency of $w$ can be estimated as follows:
\begin{equation}\label{eq:freq-w}
\dfrac{\int_{B_1}|\nabla w|^2}{\int_{\p B_1} w^2}\leq \frac{(1+\kappa)M}{\lambda^2}N.
\end{equation}

As $w$ being a solution to \eqref{eq:w}, whose coefficients of the lower order terms are bounded, the frequency control in \eqref{eq:freq-w} implies that
\begin{equation*}
\cH^{n-1}(\Gamma(w)\cap B_{1/2}) \leq C_2 \frac{(1+\kappa)M}{\lambda^2}N,
\end{equation*}
where $C_2>0$ is a constant depending only on $n$, $\lambda$ and $M$; for proofs, readers may refer to \cite{HL}. Owing to the fact that $\Gamma(w) = \Gamma(u)$, the proof is finished.
\end{proof}

%%%%%%%%%%%%%%%%%%%%%%%
%
% Section: Broken PDEs with Order $0<s\leq 1$
%
%%%%%%%%%%%%%%%%%%%%%%%

\section{Broken PDEs with Order $s>0$}\label{section:higher}

%%%%%%%%%%%%%%%%%%%%%%%
% Subsection: Problem Setting and Main Results
%%%%%%%%%%%%%%%%%%%%%%%

\subsection{Problem Setting and Main Results}\label{subsection:setting-higher}
Set $s>0$ and consider 
\begin{equation}\tag{$H$}\label{eq:higher}
\ddiv(A_s(x,u)\nabla u)= \ddiv\vv{f}(x)\quad\text{in }B_1,
\end{equation}
where
\begin{equation}\label{eq:A_s}
A_s(x,u)= a(x) + b(x)(u^+)^s,
\end{equation}
where $a$ and $b$ satisfy 
\begin{equation}\label{eq:ellip-h}
\lambda\leq a(x), b(x) \leq\frac{1}{\lambda}\quad (x\in\bar{B}_1),
\end{equation}
with a constant $0<\lambda<1$.

Before we state our main results on the case $s>0$, let us make a supplementary remark regarding the case $s<0$.

\begin{remark}\label{remark:s<0}
In the case $s<0$, the authors found it challenging even to prove the existence of solutions, provided that the coefficients $a$ and $b$ are not constants. The regularity of solutions also takes a lot of attention. To simplify the matter, let $a$ and $b$ be positive constants. Then $\phi(u;s)$, defined by 
\begin{equation*}
\phi(u;s) = \begin{dcases}
au + \frac{b(u^+)^{s+1}}{s+1},&\text{if }s<0\text{ and }s\neq -1,\\
au + b\log u^+,&\text{if }s = -1,
\end{dcases}
\end{equation*}
becomes harmonic in $B_1$. When $-1<s<0$, one may deduce from the inverse function theorem that $u\in C^{1+ s}(B_1)$.  However, when $s\leq -1$, the positive part blows up as $u$ approaches $0+$, and a set of singularities for  $\phi(u;s)$ along $\partial \{u>0 \}$
arise.
%we may only have positive solutions so as to make $\phi(u;s)$ to be harmonic. 
%The regularity of the zero level surface of $u$ will also be an interesting issue. 
\end{remark}

Our first result concerns with the optimal regularity of solutions to \eqref{eq:higher}. In what follows, we continue to use the terminologies introduced in Definition \ref{definition:fb}.

\begin{theorem}\label{theorem:reg-h} Let $u$ be a nontrivial weak solution to \eqref{eq:higher} with $|u|\leq 1$ in $\bar{B}_1$. Also set $m$ to be the greatest integer less than or equal to $s$.
\begin{enumerate}[(i)]
\item Let $0<s-m<1$ and suppose that $a,b,\vv{f}\in C^{m,s-m}(B_1)$. Then $u\in C^{m+1,s-m}(\bar{B}_{1/2})$ and 
\begin{equation*}
\norm{u}_{C^{m+1,s-m}(B_{1/2})}\leq C_s,
\end{equation*}
where $C_s$ is a positive constant depending only on $n$, $\lambda$, $s$ and the $C^{m,s-m}$ norms of $a$, $b$ and $\vv{f}$ on $\bar{B}_1$.
\item Let $s=m$ and suppose that $a,\vv{f}\in C^{m,\alpha}(\bar{B}_1)$, for some $0<\alpha<1$, and $b\in C^{m-1,1}(\bar{B}_1)$. Then $u\in C^{m,1}(B_{1/2}\cap\{u\leq 0\})$ and 
\begin{equation*}
\norm{u}_{C^{m,1}(B_{1/2}\cap\{u\leq 0\})}\leq C_{m,\alpha}.
\end{equation*}
Moreover, we have 
\begin{equation*}
|D^m u (x) - D^m u(z)|\leq C_{m,\alpha}|x-z|\quad\text{in }B_{r_{m,\alpha}}(z),
\end{equation*}
whenever $z\in\Gamma(u)\cap B_{1/2}$. Here $C_{m,\alpha}$ and $r_{m,\alpha}$ are positive constants depending only on $n$, $\lambda$, $m$, $\alpha$, the $C^{m,\alpha}$ norms of $a$ and $\vv{f}$ on $\bar{B}_1$, and the $C^{m-1,1}$ norm of $b$ on $\bar{B}_1$.
\end{enumerate}
\end{theorem}

Note that the standard regularity theory yields $u\in C^{m,\gamma}(\bar{B}_{1/2})$ (for any $0<\gamma<1$) in case (ii), since the coefficient on the nonnegative part, $\{u\geq 0\}$, is only $C^{m-1,1}$ regular. On the other hand, we know that $u$ is locally $C^{m+1,\alpha}$ on its negative part, $\{u<0\}$. What Theorem \ref{theorem:reg-h} asserts is, roughly, that $u$ behaves nicely up to the level surface where the leading coefficients experience a ``break''. 

It is worthwhile to mention that we obtain the regularity of the nondegenerate part of the zero level surface for free (via the inverse function theorem) because of the higher regularity of the derivatives of $u$. 

\begin{corollary}\label{corollary:fb-h} Let $u$ be as in Theorem \ref{theorem:reg-h}. Then $\cN(u)$ is a locally $C^{m,s}$ [resp., $C^{m,1}$] graph, whose locality depends only on $n$, $\lambda$, $\omega_0$ and the nondegeneracy constant of $u$ at the base point, provided that $0<s-m<1$ [resp., $s=m$ and $a$ is assumed as in Theorem \ref{theorem:reg-h} (ii)].
\end{corollary}

Next we state the result for the Hausdorff measure of zero level surfaces. 

\begin{theorem}\label{theorem:hauss-h} Suppose that $f\equiv 0$ and let $u$ be a nontrivial solution to \eqref{eq:higher} with $|u|\leq 1$ in $\bar{B}_1$. Then 
\begin{equation*}
\cH^{n-1}(\Gamma(u)\cap B_{1/2})\leq CN,
\end{equation*}
where
\begin{equation*}
N = \frac{\frac{3}{4}\int_{B_{3/4}}|\nabla u|^2}{\int_{\p B_{3/4}}u^2},
\end{equation*}
provided that either (i) $0<s<1$, $a,b\in C^{1,1}(\bar{B}_1)$, or (ii) $s\geq 1$ and $a,b\in C^{0,1}(\bar{B}_1)$. Here $C$ is a positive constant depending only on $n$, $\lambda$ and the $C^{1,1}$ norms [resp., $C^{0,1}$ norms] of $a$ and $b$ in case (i) [resp., case (ii)].
\end{theorem}  

%%%%%%%%%%%%%%%%%%%%%%%
% Subsection: Optimal Regularity of Solutions
%%%%%%%%%%%%%%%%%%%%%%%

\subsection{Optimal Regularity of Solutions}\label{subsection:reg-h}

Let $u$ be a (nontrivial) weak solution to \eqref{eq:higher} with $|u|\leq 1$ in $B_1$. The optimal regularity of $u$ for the case $0<s-m<1$ is rather easy. 

\begin{proof}[Proof of Theorem \ref{theorem:reg-h} (i)] Let us consider the case $0<s<1$ only, as the other case can be proved in a similar manner. Since $u$ is a bounded solution to \eqref{eq:higher}, we deduce from the De-Giorgi theory that $u\in C^\gamma(\bar{B}_{7/8})$ for a universal $0<\gamma<1$. As we define $\tilde{A}_s(x) = A_s(x,u(x))$, we obtain $\tilde{A}_s\in C^{\gamma s}(\bar{B}_{7/8})$. This in turn yields that $u\in C^{1,\gamma s}(\bar{B}_{3/4})$ by the interior H\"{o}lder estimates for gradients, and in particular, $u^+\in C^{0,1}(\bar{B}_{3/4})$. This yields $\tilde{A}_s\in C^{s}(\bar{B}_{3/4})$. Again by the interior H\"{o}lder estimates for gradients, $u\in C^{1,s}(\bar{B}_{1/2})$, as desired. 
\end{proof}

One may find in this proof that the Lipschitz regularity of $a$, $b$ and $f$ has played no crucial role. Indeed, one may weaken their regularity up to $C^s(B_1)$ class, though we omit the obvious proof.

\begin{remark}\label{remark:C1a-h} The sharpness of Theorem \ref{theorem:reg-h} can be easily deduced by considering the constant coefficient case. If $a$ and $b$ are assumed to be constant, then $\phi(u)=a u + b\frac{(u^+)^{s+1}}{s+1}$ is harmonic and, thus, smooth. However, since $\phi'(u) \geq a$ and $\phi \in C^{m+1,s-m}(\R)$ (with $0<s-m<1$ and $m\in\Z$), the inverse function theorem implies that $u$ is $C^{m+1,s-m}$ at best.
\end{remark}

Let us move on to the case $s=m$, that is, $s$ is a positive integer. 

\begin{proof}[Proof of Theorem \ref{theorem:reg-h} (ii)] Throughout this proof, we shall write $C_{m,\alpha}$ a positive constant depending only on $n$, $\lambda$, $m$, $\alpha$, the $C^{m,\alpha}$ norms of $a$ and $\vv{f}$, and the $C^{m-1,1}$ norm of $b$. We will also let it change from one line to another. 

Due to the assumption that  $a,\vv{f}\in C^{m,\alpha}(\bar{B}_1) \subset C^{m-1,1}(\bar{B}_1)$, we may follow the bootstrap argument in the proof of Theorem \ref{theorem:reg-h} (i) and derive that $u\in C^{m,\alpha}(\bar{B}_{7/8})$ and 
\begin{equation}\label{eq:Cma-u}
\norm{u}_{C^{m,\alpha}(\bar{B}_{7/8})} \leq C_{m,\alpha},
\end{equation}

Choose any $n$-dimensional multi-index $\mu$ with $|\mu|=m$, and write $w = D^\mu u$. Then $w$ satisfies 
\begin{equation}\label{eq:higher-uk}
\ddiv(\tilde{A}(x)\nabla w) = \ddiv\vv{F}(x) \quad\text{in }B_1,
\end{equation}
in the weak sense, where 
\begin{equation}\label{eq:tA}
\tilde{A}(x) = a(x) + b(x) (u^+(x))^m,
\end{equation}
and
\begin{equation}\label{eq:F}
\vv{F}(x) = -\sum_{0\leq \mu'< \mu} (D^{\mu-\mu'}\tilde{A}(x))\nabla (D^{\mu'} u)(x) + D^\mu\vv{f}(x).
\end{equation}

First we observe that $\tilde{A}\in C^{m-1,1}(\bar{B}_{7/8})\subset C^{0,1}(\bar{B}_{7/8})$ and that $\vv{F} \in L^\infty(B_{7/8})$, since all $a$, $b$, $\vv{f}$ and $u$ belong to the class $C^{m-1,1}(\bar{B}_{7/8})$. However, we shall see that $\vv{F}(x) - \vv{F}(z)$ behaves as $|x-z|^\alpha$ around any point $z\in\Gamma(u)\cap \bar{B}_{1/2}$. 

Let us rephrase $\vv{F}$ as
\begin{equation*}
\vv{F}(x) = (u^+(x))^m D^\mu b(x) \nabla u(x) + \vv{R}(x),
\end{equation*}
where the remainder term $\vv{R}$ is composed of the derivatives of $a$, $\vv{f}$ and $u$ with order at most $m$ as well as the derivatives of $b$ with order at most $m-1$. Due to the assumption that $a,\vv{f}\in C^{m,\alpha}(\bar{B}_1)$, $b\in C^{m-1,1}(\bar{B}_1)$ and the observation above that $u\in C^{m,\alpha}(\bar{B}_{7/8})$, we have 
\begin{equation}\label{eq:remainder}
\left| \vv{R}(x) - \vv{R}(z) \right| \leq C_{m,\alpha}|x-z|^\alpha,
\end{equation}
for any $x,z\in \bar{B}_{7/8}$. 

Now we fix $z\in\Gamma(u)\cap \bar{B}_{1/2}$. Then it follows from \eqref{eq:Cma-u} that
\begin{equation}\label{eq:other}
\left| (u^+(x))^m D^\mu b(x) \nabla u(x) \right| \leq C_{m,\alpha}|x-z|^m,
\end{equation}
for any $x\in B_{3/8}(z)$. Thus, combining \eqref{eq:remainder} with \eqref{eq:other}, we arrive at
\begin{equation}\label{eq:F-Ca}
\left| \vv{F}(x) - \vv{F}(z) \right| \leq C_{m,\alpha} |x-z|^\alpha,
\end{equation}
for any $x\in B_{3/8}(z)$. 
 
To this end, we may apply the classical Schauder theory (for divergence type equations) (for instance, see \cite[Theorem 8.32]{GT} and the comments below) to \eqref{eq:higher-uk} and obtain an affine polynomial $P = P_z$ such that 
\begin{equation*}
| w(x) - P(x) | \leq M_{m,\alpha} |x-z|^{1+\alpha}\quad\text{and}\quad |\nabla P|\leq N_{m,\alpha}
\end{equation*}
for any $x\in B_{r_{m,\alpha}}(z)$, where $M_{m,\alpha}$ and $N_{m,\alpha}$ are some constants depending only on universal quantities and $m$ and $\alpha$. Since $P(z) = w(z)$, we have 
\begin{equation*}
| w(x) - w(z) | \leq C_{m,\alpha} |x-z|,
\end{equation*}
for any $x\in B_{r_{m,\alpha}}(z)$, which proves the second inequality of Theorem \ref{theorem:reg-h} (ii). Arguing as in the proof of Theorem \ref{theorem:main} (i), we can verify the first assertion of Theorem \ref{theorem:reg-h} (i) as well, whence the proof is finished. 
\end{proof}

%%%%%%%%%%%%%%%%%%%%%%%
% Subsection: Hausdorff Measure of Zero Level Surfaces
%%%%%%%%%%%%%%%%%%%%%%%

\subsection{Hausdorff Measure of Nodal Sets}\label{subsection:hausdorff-h}

Unlike the regularity theory for solutions to \eqref{eq:main}, it is easier to prove the finiteness of Hausdorff measure of zero level surfaces with the case $s\geq 1$. 

\begin{proof}[Proof of Theorem \ref{theorem:hauss-h} (ii)] Suppose that $s\geq 1$. Then by Theorem \ref{theorem:reg-h}, $u^+\in C^{0,1}(\bar{B}_{3/4})$ which implies the Lipschitz continuity of the coefficient of \eqref{eq:higher}; i.e., $A_s(x,u(x)) = a(x) + b(x)(u^+(x))^s$ is Lipschitz in $x$ in $\bar{B}_{3/4}$, whose norm depends only on $n$, $\lambda$ and the Lipschitz constants of $a$ and $b$. Thus, the conclusion follows directly from \cite[Theorem 1.7]{HS}. 
\end{proof}

Now we move onto the case $0<s<1$ and $a,b\in C^{1,1}(\bar{B}_1)$, that is, the assumptions of Theorem \ref{theorem:hauss-h} (i). We shall proceed as in the proof of Theorem \ref{theorem:main-deg}. Set
\begin{equation*}
w(x) = a(x) u(x) + b(x)\frac{(u^+(x))^{s+1}}{s+1}\quad\text{in }\bar{B}_1.
\end{equation*}
Then $w$ becomes a weak solution to  
\begin{equation}\label{eq:w-h}
\Delta w = B(x)\nabla u + c(x)u\quad\text{in }B_1,
\end{equation}
where
\begin{equation*}
B(x)= \nabla a(x) + (u^+(x))^s\nabla b(x)\quad\text{and}\quad c(x)=  \Delta a(x) +\frac{(u^+(x))^s}{s+1}\Delta b(x).
\end{equation*}
It is noteworthy that under the assumption of Theorem \ref{theorem:hauss-h} (i), we have $B, c\in L^\infty(B_1)$ with 
\begin{equation}\label{eq:Bc}
\norm{B}_{L^\infty(B_1)} + \norm{c}_{L^\infty(B_1)}\leq \kappa,
\end{equation}
where $\kappa>0$ is a constant depending only on $n$, $\lambda$ and the $C^{1,1}$ norms of $a$ and $b$. 

\begin{proof}[Proof of Theorem \ref{theorem:hauss-h} (i)] 
By the structure condition \eqref{eq:ellip-h}, we have 
\begin{equation*}
\lambda |u|\leq |w|\leq \frac{2}{\lambda}|u|.
\end{equation*}
Moreover, since $\nabla w = u(\nabla a + \frac{(u^+)^s}{s+1}\nabla b) + (a+ b(u^+)^s)\nabla u$, we know from \eqref{eq:Bc} that 
\begin{equation*}
|\nabla u|\leq \Lambda(|\nabla w| + |w|),
\end{equation*}
for some $\Lambda>1$ depending only on $n$, $\lambda$ and the $C^{1,1}$ norms of $a$ and $b$.

Owing to these inequalities altogether with \eqref{eq:Bc}, we observe that the equation \eqref{eq:w-h} belongs to the class of \eqref{eq:w-alm}, so we can use all the arguments in Appendix \ref{section:almgren}. As a result, we have the almost monotonicity of Almgren's frequency formula for $w$ as in Theorem \ref{theorem:almgren}. An easy application to this is that $w$ possesses doubling property; namely, there are positive universal constants $c_0$, $c_1$, $c_2$ and $r_0$ such that for any $z\in B_{r_0/4}$ and $0<r\leq r_0/4$, 
\begin{equation*}
\int_{\p B_{2r}(z)}w^2 \leq c_0 2^{c_1N(r_0)+c_2}\int_{\p B_r(z)}w^2.
\end{equation*}

On the other hand, from Theorem \ref{theorem:reg-h} (i) we know that $u\in C^{1,s}(B_{3/4})$, whence $w\in C^{1,s}(B_{3/4})$ as well. With this regularity and the doubling property at hand, we may follow the proof of \cite[Theorem 1.7]{HS} and deduce that 
\begin{equation*}
\cH^{n-1}(\Gamma(w)\cap B_{1/2})\leq C\frac{\frac{3}{4}\int_{B_{3/4}}|\nabla w|^2}{\int_{\p B_{3/4}}w^2}.
\end{equation*}
Finally, we may recover the desired estimate on $\cH^{n-1}(\Gamma(u)\cap B_{1/2})$, by arguing similarly as in the proof of Theorem \ref{theorem:main-deg} (ii). 
\end{proof}

%%%%%%%%%%%%%%%%%%%%%%%
%
% Appendix
%
%%%%%%%%%%%%%%%%%%%%%%%

\appendix

%%%%%%%%%%%%%%%%%%%%%%%
%
% Section: Pointwise Regularity of Elliptic Equations
%
%%%%%%%%%%%%%%%%%%%%%%%

\section{Pointwise Regularity of Elliptic Equations}\label{appendix:ptwise}

First, we state a result for pointwise $L^2$-BMO estimates. 

\begin{lemma}\label{lemma:bmo} There are positive numbers $\eta$ and $C$ depending only on $n$ such that if $\vv{f}\in L^2(B_1)$ satisfies
\begin{equation}\label{eq:assump-bmo}
\fint_{B_r}\left|\vv{f}\right|^2\leq \eta^2,
\end{equation}
for any $0<r\leq 1$, then a weak solution $v$ of $\Delta v = \ddiv\vv{f}$ in $B_1$ satisfies $L^2$-BMO estimates at the origin; that is, 
\begin{equation}\label{eq:bmo}
\fint_{B_r}|\nabla v - (\nabla v)_r|^2\leq C\fint_{B_1}v^2,
\end{equation}
for any $0<r\leq \frac{1}{2}$, where $(\nabla v)_r=\fint_{B_r}\nabla v$.
\end{lemma}

\begin{proof} By a normalization argument, one may assume without loss of generality that $\fint_{B_r}v^2\leq 1$. 

\begin{claim*} There exist a positive number $\rho$, depending only on $n$, and a sequence $\{P_l\}_{l=1}^\infty$ of polynomials with degree $1$ such that 
\begin{equation}\label{eq:claim-bmo}
\fint_{B_{\rho^l}}|v - P_l|^2 \leq \rho^{2l},\quad l=0,1,2,\cdots.
\end{equation}
\end{claim*}

Notice that the initial case is simply established by choosing $P_0\equiv 0$. Now suppose that there is an affine function $P_k$ for some $k\geq 1$ such that \eqref{eq:claim-bmo} is satisfied for $l=k$. 

Define 
\begin{equation*}
v_k(x)=\frac{1}{\rho^k}(v-P_k)(\rho^kx)\quad\text{in }B_1.
\end{equation*}
Observe that $v_k$ is a weak solution of 
\begin{equation}\label{eq:vk-bmo}
\Delta v_k=\ddiv\vv{f_k}\quad\text{in }B_1,
\end{equation}
where $\vv{f_k}(x)=\vv{f}(\rho^k x)$. From \eqref{eq:assump-bmo} we know that 
\begin{equation}\label{eq:fk-bmo} 
\fint_{B_1}\left|\vv{f_k}\right|^2 = \fint_{B_{\rho^k}}\left|\vv{f}\right|^2\leq \eta^2.
\end{equation}

Let $h$ be a harmonic replacement of $v_k$ in $B_1$. That is, $h$ is a harmonic function in $B_1$ such that $h=v_k$ on $\p B_1$. Then since $\Delta(v_k-h)=\ddiv\vv{f_k}$ in $B_1$ with $v_k-h\in W_0^{1,2}(B_1)$, we obtain 
\begin{equation*}
\norm{\nabla (v_k-h)}_{L^2(B_1)}\leq \norm{\vv{f_k}}_{L^2(B_1)}\leq c_1\eta,
\end{equation*}
where $c_1$ depends only on $n$. The Gagliardo-Nirenberg-Sobolev inequality tells us that
\begin{equation*}
\norm{v_k-h}_{L^{\frac{2n}{n-2}}(B_1)}\leq c_2\eta,
\end{equation*}
where $c_2$ depends only on $n$. Notice that $\frac{2n}{n-2}$ is the Sobolev conjugate of $2$. The last inequality combined with the Schwartz inequality in turn implies that
\begin{equation}\label{eq:vk-h-bmo}
\fint_{B_\rho}|v_k-h|^2 \leq \left(\int_{B_1}|v_k-h|^{\frac{2n}{n-2}}\right)^{1-\frac{2}{n}}|B_\rho|^{\frac{2}{n}-1}\leq c_3\eta^2\rho^{2-n},
\end{equation}
where $c_3$ depends only on $n$. 

Define $\tilde{P}_k(x)= h(0) + \nabla h(0)\cdot x$. By the interior gradient estimate for harmonic functions, it is not hard to see that 
\begin{equation}\label{eq:h-tlk-bmo}
\norm{h-\tilde{P}_k}_{L^\infty(B_\rho)}\leq c_4\rho^2,
\end{equation}
for any $0<\rho\leq \frac{1}{2}$, where $c_4$ depends only on $n$. Using \eqref{eq:vk-h-bmo} and \eqref{eq:h-tlk-bmo}, we deduce that
\begin{equation}\label{eq:vk-tlk-bmo}
\fint_{B_\rho}|v_k-\tilde{P}_k|^2 \leq \fint_{B_\rho}|v_k-h|^2 + \fint_{B_\rho}|h-\tilde{P}_k|^2\leq c_3\eta^2\rho^{2-n} + c_4^2\rho^4\leq \rho^2,
\end{equation}
provided that we have chosen $\rho$ such that $c_4^2\rho^2\leq\frac{1}{2}$ and accordingly $\eta$ so that $c_3\eta^2\leq \frac{1}{2}\rho^n$.

Finally we define a first order polynomial $P_{k+1}$ by 
\begin{equation*}
P_{k+1}(x) = P_k(x) + \rho^k\tilde{P}_k(\rho^{-k}x),
\end{equation*}
by which one may easily verify that \eqref{eq:claim-bmo} holds with $l=k+1$. The proof of the claim is now finished by the induction principle. 

With the presence of sequence $\{P_l\}_{l=1}^\infty$ satisfying \eqref{eq:claim-bmo}, it is now easy to prove \eqref{eq:bmo}. Indeed, as $v_k$ satisfies $\fint_{B_1}v_k^2\leq 1$ and \eqref{eq:vk-bmo} in $B_1$, we obtain from the local energy estimate that 
\begin{equation*}
\fint_{B_{1/2}}|\nabla v_k|^2 \leq C,
\end{equation*}
where $C$ depends only on $n$. Rescaling and using the fact that $\nabla P_k$ is a constant vector, we arrive at
\begin{equation*}
\fint_{B_{\rho^k}{2}}|\nabla v-(\nabla v)_{\rho^k}|^2 \leq 2\fint_{B_{\rho^k}}|\nabla v - \nabla P_k|^2 \leq 2C,
\end{equation*} 
finishing the proof. 
\end{proof}

Next we give a pointwise $C^{d,\alpha}$ regularity.

\begin{lemma}\label{lemma:C1a} Given $0<\alpha<1$ and $d\geq 1$ an integer, there is a positive number $\eta$ depending only on $n$, $d$ and $\alpha$ such that if $\vv{f}\in L^p(B_1)$ for some $p>n$ satisfying
\begin{equation}\label{eq:assump-f}
\left(\fint_{B_r}\left|\vv{f}\right|^p\right)^{\frac{1}{p}}\leq \eta r^{d-1+\alpha},
\end{equation}
for any $0<r\leq 1$, then a weak solution $v$ of $\Delta v = \ddiv\vv{f}$ in $B_1$ with $\norm{v}_{L^\infty(B_1)}\leq 1$ is $C^{d,\alpha}$ at the origin. That is, there is a homogeneous harmonic polynomial $P$ of degree $d$ such that 
\begin{equation}\label{eq:v-P}
|v(x)-P(x)|\leq M_d|x|^{d+\alpha},
\end{equation}
for any $x\in B_1$, and 
\begin{equation}\label{eq:v-P-Db}
\sum_{|\beta|\leq d}|D^\beta P(0)|\leq N_d.
\end{equation}
Moreover, we have 
\begin{equation}\label{eq:grad(v-P)}
\left(\fint_{B_r}|\nabla (v-P)|^p\right)^{\frac{1}{p}}\leq K_dr^{d-1+\alpha},
\end{equation}
for any $0<r\leq \frac{1}{2}$. Here $r_d$, $M_d$, $N_d$ and $K_d$ depend only on $n$, $d$, $\alpha$ and $p$. 
\end{lemma}

\begin{proof} As $\vv{f}$ is integrable with a higher exponent, we may use an enhanced approximation lemma, which is stated as follows: Given $\e>0$, if $\vv{f}\in L^p(B_1)$ for $p>n$ satisfying
\begin{equation*}
\left(\fint_{B_1}|\vv{f}|^p\right)^{1/p}\leq \e,
\end{equation*}
then for any weak solution $v$ to $\Delta v = \ddiv\vv{f}$ in $B_1$ with $\norm{v}_{L^\infty(B_1)}\leq 1$,
there exists a harmonic function $h$ in $B_1$ such that 
\begin{equation*}
\norm{v-h}_{L^\infty(B_1)}\leq C\e,
\end{equation*}
where $C$ depends only on $n$. Hence, we may use this lemma instead of local energy estimates to proceed with the induction principle. We leave out the details to the reader, since they mainly follow the proof of Lemma \ref{lemma:bmo}.
\end{proof}

%%%%%%%%%%%%%%%%%%%%%%%
%
% Section: Almgren's Monotonicity Formula
%
%%%%%%%%%%%%%%%%%%%%%%%

\section{Almgren's Monotonicity Formula}\label{section:almgren}

Let $u\in W^{1,2}(B_1)$ and suppose that $w\in W^{1,2}(B_1)$ is a function satisfying  
\begin{equation}\label{eq:u:w}
\lambda |u| \leq |w| \leq \frac{1}{\lambda} |u|\quad\text{and}\quad |\nabla u|\leq \Lambda(|\nabla w| + |w|)\quad\text{a.e. in }B_1,
\end{equation}
for some constants $0<\lambda<1$ and $\Lambda>1$. 

We are interested in the case when $w$ satisfies 
\begin{equation}\label{eq:w-alm}
- \Delta w + \vv{b}\nabla u + c u=0\quad\text{in }B_1,
\end{equation}
in the weak sense, where $\vv{b},c\in L^\infty(B_1)$ with 
\begin{equation*}
\norm{\vv{b}}_{L^\infty(B_1)} + \norm{c}_{L^\infty(B_1)}\leq \kappa.
\end{equation*}

Throughout this section, any constant $C$ is assumed to be positive and rely only on parameters $n$, $\lambda$, $\Lambda$ and $\kappa$, unless otherwise stated. Also by $C_f$ we will denote a universal constant depending further but at most on the parameter $f$.

Set  
\begin{align*}
H(r) &= \int_{\p B_r}w^2,\quad I(r) = \int_{B_r}(|\nabla w|^2 + w\vv{b} \nabla u + cwu),
\end{align*}
and 
\begin{equation*}
N(r) = \frac{rI(r)}{H(r)}.
\end{equation*}
We call $N(r)$ Almgren's frequency formula and shall prove that it is almost monotone in $r$, as stated below.

\begin{theorem}\label{theorem:almgren} Let $w$ be a weak solution of \eqref{eq:w-alm} in $B_1$. Then we have
\begin{equation}\label{eq:almgren}
N(r) \leq c_1 + c_2 N(r_0),
\end{equation}
for any $0<r\leq r_0$.
\end{theorem}

The proof of this theorem follows closely to that of \cite[Theorem 4.1]{GL}, and hence, we are only going to mention the modification and omit the redundant details. 

It is important to notice that 
\begin{equation}\label{eq:basic-W12}
\int_{B_r}w^2 \leq \frac{2r}{n}\int_{\p B_r}w^2 + \frac{4r^2}{n^2}\int_{B_r}|\nabla w|^2,
\end{equation}
which is indeed true for any element in $W^{1,2}(B_1)$.

\begin{lemma}\label{lemma:D:IH} Under the assumption of Theorem \ref{theorem:almgren}, we have 
\begin{equation}\label{eq:D:IH}
\left|\int_{B_r}w(\vv{b}\nabla u + cu)\right|\leq \frac{1}{2}\int_{B_r}|\nabla w|^2 + CrH(r),
\end{equation}
for any $0<r\leq r_0$.
\end{lemma}

\begin{proof} By a generalized Young's inequality, \eqref{eq:u:w} and \eqref{eq:basic-W12}, one may prove that 
\begin{equation*}
\left|\int_{B_r}w(\vv{b}\nabla u + cu)\right| \leq \left(\frac{\e}{\lambda^2} + C_\e r^2\right)\int_{B_r}|\nabla w|^2 + C_\e rH(r),
\end{equation*}
where $C_\e$ is a positive constant depending only on $n$, $\lambda$, $\Lambda$, $\kappa$ and $\e$. Set $\e = \lambda^2/4$, and correspondingly choose $r_0\leq 1$ such that $4C_{\lambda^2/4}r_0\leq 1$, which will make the proof complete. 
\end{proof}

An immediate corollary is that for a nontrivial solution $w$ to \eqref{eq:w-alm}, we have 
\begin{equation*}
H(r) \neq 0\quad\text{for any }0<r\leq r_0,
\end{equation*}
where $r_0$ is chosen as in Lemma \ref{lemma:D:IH}, which also implies that $r\mapsto N(r)$ is absolutely continuous in $(0,r_0]$. Thus, the set $\Omega_{r_0}$, defined by,
\begin{equation*}
\Omega_{r_0} =  \{r\in r_0 : N(r) > \max\{1,N(r_0)\}\},
\end{equation*} 
is an open subset of $\R$, which can be represented in a countable disjoint intervals.

\begin{corollary}\label{corollary:D:IH} For any $r\in\Omega_{r_0}$, 
\begin{equation}\label{eq:D:I}
\int_{B_r}|\nabla w|^2 \leq CI(r)\quad\text{and}\quad \int_{B_r}w^2 \leq Cr^2I(r).
\end{equation}
\end{corollary}

\begin{proof} Notice that if $r\in \Omega_{r_0}$, then $N(r)>1$, which is equivalent to that $H(r) < rI(r)$. The rest of the proof follows easily from \eqref{eq:D:IH}, \eqref{eq:basic-W12}.
\end{proof}

In order to use the Rellich formula, we need $W^{2,2}$ regularity of $w$. 
\begin{lemma}\label{lemma:W22} One has $w\in W_{loc}^{2,2}(B_1)$ and 
\begin{equation*}
\norm{w}_{W^{2,2}(V)}\leq C_{V}\norm{w}_{W^{1,2}(B_1)},
\end{equation*}
for any compact subset $V$ of $B_1$.
\end{lemma}

\begin{proof} The proof follows by the standard Bernstein technique, which we leave out to the reader.
\end{proof}

We are now ready to proof the almost monotonicity of Almgren's frequency formula.

\begin{proof}[Proof of Theorem \ref{theorem:almgren}] By the definition of $\Omega_{r_0}$, we may only prove the statement for $r\in\Omega_{r_0}$.

Let us proceed with the computaton of $I'(r)$. By Lemma \ref{lemma:W22}, we may use the Rellich formula on the computation of $\int_{\p B_r}|\nabla w|^2$, and owing to \eqref{eq:w-alm}, we arrive at
\begin{equation}\label{eq:I'}
\begin{split}
I'(r) &= \frac{n-2}{r}I(r) + 2\int_{\p B_r}w^2 \\
&\quad + \int_{\p B_r}w(\vv{b}\nabla u + cu) - \frac{n}{r}\int_{B_r}w(\vv{b}\nabla u + cu).
\end{split}
\end{equation}
It follows from \eqref{eq:D:IH} and \eqref{eq:D:I} that
\begin{equation}\label{eq:I'-1}
I'(r) = \left(\frac{n-2}{r}+O(1)\right)I(r) + 2\int_{\p B_r}w_r^2 + \int_{\p B_r}w(\vv{b}\nabla u + cu),
\end{equation}
where by $O(1)$ we denote a quantity which is bounded by a constant depending only on $n$, $\lambda$, $\Lambda$ and $\kappa$.

In the sequel, one may follow the proof of \cite[Theorem 4.1]{GL}, as long as we have 
\begin{equation}\label{eq:wb}
\begin{split}
\left|\int_{\p B_r}w(\vv{b}\nabla u + cu)\right|\leq  CI(r) + C\left(\int_{\p B_r}w^2\int_{\p B_r}w_r^2\right)^{1/2}.
\end{split}
\end{equation}
However, it is an easy consequence of the choice of $r\in \Omega_{r_0}$. Indeed, from the weak formulation of \eqref{eq:w-alm}, we have 
\begin{equation*}
I(r)^2 = \left(\int_{\p B_r} ww_r\right)^2 \leq \int_{\p B_r}w^2\int_{\p B_r}w_r^2 \leq rI(r) \int_{\p B_r}w_r^2,
\end{equation*}
where in deriving the last inequality we have used that $H(r) \leq rI(r)$ for $r\in\Omega_{r_0}$. Inserting this inequality into \eqref{eq:I'-1}, we arrive at
\begin{equation}\label{eq:Dw:wr}
\int_{\p B_r}|\nabla w|^2 \leq C\int_{\p B_r}w_r^2\quad (0< r\in\Omega_{r_0}),
\end{equation}
from which \eqref{eq:wb} follows immediately. 

In order to reduce the redundancy to the existing literature, we skip the details of what is left in the proof.
\end{proof}


\begin{thebibliography}{}

\bibitem[A]{A} Aramaki, J. {\it Estimate of the Hausdorff measure of the singular set of a solution for a semi-linear elliptic equation associated with superconductivity}. Arch. Math. (Brno) {\bf 46}(3) (2010), 185-201. 
\bibitem[B]{B} Ber, L. {\it Local behavior of solutions of general linear elliptic equations}. Commun. Pure Appl. Math. {\bf 8}(4) (1955), 473-496.  
\bibitem[CF]{CF} Caffarelli, L.A.; Friedman, A. {\it Partial regularity of the zero-set of solutions of linear and superlinear elliptic equations}. J. Differ. Equ. {\bf 60} (1985), 420-433.
\bibitem[CNV]{CNV} Cheeger, J.; Naber, A.; Valtorta, D. {\it Critical sets of elliptic equations}. Commun. Pure Appl. Math. {\bf 68}(2) (2015), 173-209.
\bibitem[GL]{GL} Garofalo, N.; Lin, F.-H. {\it Unique continuation for elliptic operators: a geometric-variational approach}. Commun. Pure Appl. Math. {\bf 40} (1987), 347-366.
\bibitem[GT]{GT} Gilbarg, D.; Trudinger, N. S. {\it Elliptic partial differential equations of second order}. Springer-Verlag, Berlin-Heidelberg (1983)
\bibitem[H]{H} Han, Q. {\it Singular sets of solutions to elliptic equations}. Indiana Univ. Math. J. {\bf 43}(3) (1994), 983-1002.
\bibitem[HL]{HL} Han, Q.; Lin, F. H. {\it Nodal sets of solutions of elliptic partial differential equations}. Books available on Han?s homepage (2013).
\bibitem[HHHN]{HHHN} Hardt, R.; Hoffman-Ostenhof, M.; Hoffman-Ostenhof, T.; Nadirashvili, N. {\it Critical sets of solutions of elliptic equations}. J. Differ. Geom. {\bf 51} (1999), 359-373. 
\bibitem[HS]{HS} Hardt, S.; Simon, L. {\it Nodal sets for solutions of elliptic equations}. J. Differ. Geom. {\bf 30} (1989), 505-522.
\bibitem[HHN]{HHN} Hoffman-Ostenhof, M.; Hoffman-Ostenhof, T.; Nadirashvili, N. {\it Critical sets of smooth solutions of elliptic equations in dimension 3}. Indiana Univ. Math. J. {\bf 45} (1996), 15-37. 
\bibitem[KLS]{KLS} Kim, S.; Lee, K.; Shahgholian, H. {\it An elliptic free boundary arising from the jump of conductivity}. arXiv:1605.06558v2 (2016)
\bibitem[L]{L} Lin, F.-H. {\it Nodal sets of solutions of elliptic and parabolic equations}. Commun. Pure Appl. Math. {\bf 44} (1991), 287-308.
\bibitem[M]{M} Miller, K. {\it Nonunique continuation for uniformly parabolic and elliptic equations in self-adjoint divergence form with Holder continuous coefficients}. Arch. Ration. Mech. Anal. {\bf 54}(2) (1974),105-117.
\bibitem[P]{P} Pli\'{s}, E. {\it On non-uniqueness in Cauchy problem for an elliptic second order differential equation}. Bull. Acad. Polonaise Sci. {\bf 21} (1960), 1-92. 
  
  
\end{thebibliography}
\end{document}